\documentclass[onefignum,onetabnum]{siamart171218}

\usepackage{lipsum}
\usepackage{amsfonts}
\usepackage{amssymb}
\usepackage{graphicx}
\usepackage{epstopdf}
\usepackage{mathtools}
\usepackage{booktabs}
\usepackage{tikz}
\usepackage{algorithmic}
\ifpdf
  \DeclareGraphicsExtensions{.eps,.pdf,.png,.jpg}
\else
  \DeclareGraphicsExtensions{.eps}
\fi

\newsiamremark{remark}{Remark}
\newsiamremark{assumption}{Assumption}
\newsiamremark{hypothesis}{Hypothesis}
\newsiamremark{example}{Example}
\crefname{hypothesis}{Hypothesis}{Hypotheses}
\newsiamthm{claim}{Claim}

\title{Linearly convergent adjoint free solution of least squares problems by random descent
\thanks{
\funding{This work has received funding from the European Union’s Framework Programme for Research and Innovation Horizon 2020 (2014-2020) under the Marie Sk\l odowska-Curie Grant Agreement No. 861137.}}}

\author{
Dirk A. Lorenz\thanks{Institute of Analysis and Algebra, TU Braunschweig, Braunschweig, Germany 
  (\email{d.lorenz@tu-braunschweig.de}, \url{https://www.tu-braunschweig.de/en/iaa/personal/lorenz}).}
\and Felix Schneppe\thanks{Institute of Analysis and Algebra, TU Braunschweig, Braunschweig, Germany 
  (\email{f.schneppe@tu-braunschweig.de}, \url{https://www.tu-braunschweig.de/iaa/personal/schneppe}).}
  \and Lionel Tondji\thanks{Institute of Analysis and Algebra, TU Braunschweig, Braunschweig, Germany 
  (\email{l.ngoupeyou-tondji@tu-braunschweig.de}, \url{https://www.tu-braunschweig.de/iaa/personal/l-ngoupeyou-tondji}).}
}

\usepackage{amsopn}
  
\newcommand{\RR}{\mathbb{R}}

\newcommand{\EE}{\mathbb{E}}
\newcommand{\PP}{\mathbb{P}}
\renewcommand{\SS}{\mathbb{S}}
\newcommand{\cN}{\mathcal{N}}
\newcommand{\scp}[2]{\langle{#1},{#2}\rangle}
\newcommand{\rg}{\operatorname{rg}}
\newcommand{\dd}{\operatorname{d}}
\newcommand{\mleq}{\preccurlyeq}
\newcommand{\mgeq}{\succcurlyeq}
\DeclarePairedDelimiter{\norm}{\|}{\|}

\DeclareMathOperator{\Unif}{Unif}
\DeclareMathOperator{\diag}{diag}
\DeclareMathOperator{\trace}{trace}
\DeclareMathOperator{\rank}{rank}

\makeatletter
\newcommand*{\addFileDependency}[1]{
  \typeout{(#1)}
  \@addtofilelist{#1}
  \IfFileExists{#1}{}{\typeout{No file #1.}}
}
\makeatother

\graphicspath{ {./figures/} }
\usepackage{subfig}

\ifpdf
\hypersetup{
  pdftitle={},
  pdfauthor={Dirk A. Lorenz, Felix Schneppe, Lionel Tondji}
}
\fi

\begin{document}

\maketitle

\begin{abstract}
  We consider the problem of solving linear least squares problems in a framework where only evaluations of the linear map are possible. We derive randomized methods that do not need any other matrix operations than forward evaluations, especially no evaluation of the adjoint map is needed. Our method is motivated by the simple observation that one can get an unbiased estimate of the application of the adjoint. We show convergence of the method and then derive a more efficient method that uses an exact linesearch. This method, called random descent, resembles known methods in other context and has the randomized coordinate descent method as special case. We provide convergence analysis of the random descent method emphasizing the dependence on the underlying distribution of the random vectors. Furthermore we investigate the applicability of the method in the context of ill-posed inverse problems and show that the method can have beneficial properties when the unknown solution is rough. We illustrate the theoretical findings in numerical examples. One particular result is that the random descent method actually outperforms established transposed-free methods (TFQMR and CGS)  in examples.
\end{abstract}

\begin{keywords}
  randomized algorithms, least squares problems, stochastic gradient descent
\end{keywords}

\begin{AMS}
  65F10,
  68W20,
  15A06
\end{AMS}

\section{Introduction}
\label{sec:intro}

\subsection{Problem statement}
\label{sec:prob-statement}

We consider the basic problem of solving linear least squares problems 
\begin{align}\label{eq:least-squares}
\min_{v\in\RR^{d}}\tfrac12\norm{Av-b}^{2}
\end{align}
with a linear map $A$ from $\RR^{d}$ to $\RR^{m}$ and $b\in\RR^{m}$. We revisit this classical problem under the following restricting assumptions:
\begin{itemize}
\item $A$ is not given as a matrix, but only an implementation is available that takes as input a vector $v\in\RR^{d}$ and returns the output $Av\in\RR^{m}$.
\item It is not feasible to store many vectors of either size $d$ or $m$ due to memory constraints (especially we assume that it is not possible to build large parts of the full matrix by computing $Ae_{i}$ for all standard basis vectors $e_i$).
\item The adjoint linear map (i.e. the transposed matrix $A^{T}$) can not be evaluated. (One can, in principle, evaluate the adjoint map by applying automatic differentiation to the map $x\mapsto\scp{y}{Ax}$. However, this requires that the programming language allows for automatic differentiation and that the program is written in a suitable way, which, however, is not always the case.)
\end{itemize}
We do not make any structural assumptions on  the linear map $A$. Especially we consider the square, over-, and under-determined cases and the cases of full rank and rank deficient maps.
If the evaluation of $A^{T}$ is not possible, one can not immediately calculate the gradient of the least square function, as this also involves the application of the adjoint linear map.

In this paper we will derive randomized methods which do provably converge to a solution of the least squares functional (in expectation) under our restrictive assumptions formulated above.
The idea behind these algorithms is a simple observation: There is a simple way to get an unbiased sample of the gradient $A^{T}(Av-b)$ which only uses evaluations of $A$:
\begin{lemma}\label{lem:sampled-gradient}
  Let $x\in\RR^{d}$ be a random vector such that $\EE(xx^{T}) = I_{d}$. Then it holds that 
\begin{align*}
\EE(\scp{Av-b}{Ax}x) = A^{T}(Av-b),
\end{align*}
in other words, $\scp{Av-b}{Ax}x$ is an unbiased estimator for the gradient of $\tfrac12\norm{Av-b}^{2}$.
\end{lemma}
\begin{proof}
  We rewrite $\scp{Av-b}{Ax}x = x\scp{Ax}{Av-b} = xx^{T}A^{T}(Av-b)$ and see that the result follows from linearity of the expectation.
\end{proof}

This observation motivates the use of the stochastic gradient method for the minimization of the least squares functional~(\ref{eq:least-squares}) which we state in Algorithm~\ref{alg:sgd}. We will analyze the convergence of this algorithm in Section~\ref{sec:sgd} and especially treat the question of suitable choices of the stepsizes $\tau_{k}$.

\begin{algorithm}
  \caption{Stochastic gradient descent with adjoint sampling (SGDAS)}
  \label{alg:sgd}
  \begin{algorithmic}
    \STATE Initialize $v^{0}=0\in\RR^{d}$, $k=0$
    \WHILE{not stopped}
    \STATE obtain random vector $x\in\RR^{d}$ with $\EE(xx^{T}) = I_{d}$
    \STATE update $v^{k+1} = v^{k} - \tau_{k}\scp{Av^{k}-b}{Ax}x$ for some stepsize $\tau_{k}$
    \STATE update $k\leftarrow k+1$
    \ENDWHILE 
  \end{algorithmic}
\end{algorithm}

One observes that the SGDAS method takes a step in a fairly random direction $x$ (it is a simple observation that the coefficient $\scp{Av^{k}-b}{Ax}$ turns this direction into a descent direction). This motivates to consider an even simpler method where we omit the coefficient and let the choice of the stepsize do the work to guarantee descent. This method is given in Algorithm~\ref{alg:rd}.

\begin{algorithm}
  \caption{Random descent (RD)}
  \label{alg:rd}
  \begin{algorithmic}
    \STATE Initialize $v^{0}=0\in\RR^{d}$, $k=0$
    \WHILE{not stopped}
    \STATE obtain random vector $x\in\RR^{d}$
    \STATE update $v^{k+1} = v^{k} + \tau_{k}x$ where the stepsize solves $\min\limits_{\tau}\tfrac12\norm{A(v^{k}+\tau x) - b}^{2}$
    \STATE update $k\leftarrow k+1$
    \ENDWHILE 
  \end{algorithmic}
\end{algorithm}

\subsection{Related work}
\label{sec:related-work}

Methods for the solution of linear least squares problems or linear equations that only need evaluations of $A$ are known as ``transpose-free methods'' and an early such method for the case of general square linear systems of equations has been proposed Freund in \cite{freund1993tfqmr} as a generalization of the quasi-minimal residual method named \emph{transpose free quasi-minimal residual method} (TFQMR).\footnote{Note that methods for \emph{symmetric} linear systems often do not make use of the transpose as in this case $Ax-b$ is already the gradient of the objective $x^{T}Ax - x^{T}b$.} Slightly later Brezinski and Redivo-Zaglia~\cite{brezinski1998transpose} and Chan et al.~\cite{chan1998transpose} proposed a method of Lanczos type which lead to the \emph{conjugate gradient squared method} (CGS) proposed by Sonneveld~\cite{sonneveld1989cgs}.

These mentioned methods are available in current software (e.g. in Scipy) and are deterministic. Early randomized methods, again for square linear systems, are based on the idea of reformulating the system $Av=b$ as $v = Gb + (I-GA)v$ for some square matrix $G$ and then consider the solution expressed by the Neumann series 
\begin{align*}
v = \sum\limits_{r=0}^{\infty}(I-GA)^{r}Gb,
\end{align*}
which converges to the solution if the spectral radius of $I-GA$ is smaller than one. Then an approximate solution is obtained by using a specific random walk to approximate the truncated series, see \cite{forsythe1950mc,curtiss1954mc,halton1994sequential,dimov2000mc,wang2008mc}.

Gradient based methods, however, are usually not applicable since the evaluation of the gradient of $\norm{Av-b}^{2}$ needs the evaluation of $A^{T}$. One exception are coordinate descent methods~\cite{leventhal2010randomized} which will be a special case of the random descent method. Another related method is the \emph{random search} method (see~\cite[Section 3.4]{polyak1987}). This is a derivative free method for the minimization of a convex objective $\Phi:\RR^{d}\to \RR$ and the iterates are 
\begin{align*}
v^{k+1} = v^{k} - \tfrac{\gamma_k}{\alpha_k}\left[ \Phi(v^{k}+\alpha_{k} u) - \Phi(v^{k}) \right]u
\end{align*}
for stepsizes $\gamma_{k}$ and discretization length $\alpha_{k}$ and where $u$ is sampled uniformly from the unit sphere. As discussed in~\cite{polyak1987}, this method converges if $\Phi$ has Lipschitz continuous gradient, $\gamma_{k}$ is small enough and $\alpha_{k}\to 0$. Applied to the objective $\Phi(v) = \tfrac12\norm{Av-b}^{2}$ we obtain 
\begin{align*}
  v^{k+1} & = v^k - \tfrac{\gamma_k}{\alpha_k}\left[ \tfrac12\norm{Av^{k}+\alpha_{k}Au - b}^{2} - \tfrac12\norm{Av^{k}-b}^{2} \right]u\\
  & = v^{k} - \gamma_{k}  \scp{Av^{k}-b}{Au}u - \tfrac{\gamma_k\alpha_{k}}2\norm{Au}^{2}u
\end{align*}
and if we set $\alpha_{k}=0$ in the final expression we obtain the update Algorithm~\ref{alg:sgd} (but with $u$ being uniformly distributed on the unit sphere). While convergence of the random search method is known, we are not aware of known rates for this method and, to the best of our knowledge, only the case of sampling from the unit sphere has been analyzed. An extension of random search with exact and inexact linesearch (similar to Algorithm~\ref{alg:rd}) has been studied in~\cite{stich2013optimization}. There, the authors obtain convergence rates in the strongly convex case for directions $x$ uniformly distributed on the unit sphere and random coordinate vectors. Later~\cite{nesterov2017gradient-free} studied a related method based on smoothing: For $\mu>0$ the authors consider 
\begin{align*}
f_{\mu}(x) = \EE_{u}(f(x+\mu u))
\end{align*}
which amounts to a convolution of $f$ with the scaled probability density of $u$ (in~\cite{nesterov2017gradient-free} the authors consider normally distributed search directions). It is observed that 
\begin{align*}
\nabla f_{\mu}(x) = \tfrac1\mu \EE((f(x+\mu u) - f(x))u)
\end{align*}
i.e. that the expectation of the random search step is a gradient step for a smoothed function. The authors derive convergence rates and also an accelerated method.

Our method can also be derived from the framework by Gower and Richtarik~\cite{gower2015randomized}: Their class of methods read as 
\begin{align}\label{eq:sketch-and-project}
v^{k+1} = v^{k} - B^{-1}A^{T}S(S^{T}AB^{-1}A^{T}S)^{\dag}S^{T}(Av^{k}-b)
\end{align}
where $S\in\RR^{l\times m}$ is a random sketching matrix and $B\in\RR^{d\times d}$ is a positive definite matrix. If $A$ has linearly independent columns one can set $B = A^{T}A$ and use $S=Ax\in \RR^{m}$ with $x\in\RR^{d}$ being a random vector one ends up with the random descent method.

Finally, note that SGDAS and RD are fundamentally different from the randomized Kaczmarz method~\cite{strohmer2009kaczmarz} (which can also be obtained from~\eqref{eq:sketch-and-project} by choosing $B=I_{d}$ and $S=e_{k}$ being a random coordinate vector). The iterates of the Kaczmarz iteration always operate in the range of $A^{T}$ while SGDAS and RD operate in the full space.

\subsection{Contribution}
\label{sec:contribution}

\begin{itemize}
\item 
  We derive transpose free methods that minimize the linear least squares functional which only use applications of $A$ (i.e. no applications of $A^{T}$ are done and no other data like knowledge of $\norm{A}$ is used). These methods can be seen as special cases of previously known random search methods, but we obtain more refined convergence results.
\item 
  We illustrate that these adjoint free methods can outperform established methods for least squares problems like TFQMR and CGS.
\item We do a detailed analysis of the random descent method (Algorithm~\ref{alg:rd}) for a class of distributions of the search direction and obtain results on the influence of the distribution of the search direction on the rate.
\item Without any assumptions on $A$ we show that the residual of the normal equations converges (in expectation) at a sublinear rate.
\item 
  We analyze the semi-convergence property of the method in the case of inconsistent problems and illustrate that random descent may be beneficial (in comparison to other iterative methods like the Landweber method) for the solution of linear ill-posed inverse problems when the solution is rough.
\end{itemize}

\section{Notation and basic results}
\label{sec:notation}

We use $\norm{x}$ to denote the Euclidean norm, write $\rho(M)$ for the spectral radius of a square matrix $M$, i.e. the maximum magnitude of the eigenvalues of $M$. With $\lambda_{\max}(M)$ and $\lambda_{\min}(M)$ we denote the largest and smallest eigenvalue of a real symmetric matrix $M$, respectively and $\sigma_{\min}(A)$ is the smallest positive singular value of a matrix $A$. The range of matrix $A$ is denoted by $\rg(A)$ and the kernel by $\ker(A)$. Moreover, $\EE$ denotes the expectation over all involved randomness (if not stated otherwise). With $A\mgeq B$ we denote the Loewner order on symmetric matrices, i.e. $A\mgeq B$ if $A-B$ is positive semi-definite.

\section{Convergence results}
\label{sec:convergence}

We start our convergence analysis of SGDAS and RD by stating the standing assumption on the random directions.

\begin{assumption}
\label{ass:distribution}
We assume that our random vector $x \in \RR^d$ is \emph{isotropic}, i.e. that it fulfills $\EE(xx^T) = I_{d}$. Moreover, we assume that  and $\EE(xx^{T}\norm{x}^{2}) = cI_{d}$ for some $c$.
\end{assumption}

By linearity of the expectation we get
\begin{align*}
\EE(\norm{x}^{2}) =\EE(x^{T}x) = \EE(\trace(xx^{T})) = \trace(\EE(xx^{T})) = d
\end{align*}
for random vectors satisfying Assumption~\ref{ass:distribution} (cf.~\cite[Lemma 3.2.4]{vershynin2018high}). The further assumption that $\EE(xx^{T}\norm{x}^{2})$ is a multiple of the identity is needed in the proofs below and is fulfilled in the list of sampling methods in the next example.

\begin{example}
  Examples for random vectors that fulfill Assumption~\ref{ass:distribution} are:
  \begin{enumerate}
  \item Vectors that are \textbf{uniformly distributed on the sphere} with radius $\sqrt{d}$, i.e. 
\begin{align*}
x\sim \Unif(\sqrt{d}\SS^{d-1}).
\end{align*}
Here we have $\norm{x}^{2}=d$ deterministically and hence $\EE(xx^{T}\norm{x}^{2}) = dI_{d}$ which means that $c=d$.
\item \textbf{Normally distributed} random vectors, i.e. 
\begin{align*}
x\sim \cN(0,I_{d}).
\end{align*}
In this case one calculates $c = \EE(xx^{T}\norm{x}^{2}) = d+2$.
  \item \textbf{Rademacher vectors}, i.e. vectors with entries which are $\pm1$, each with probability $1/2$ i.e. $\PP(x_{k}= \pm 1) = \tfrac12$, independently. Here we also have $\norm{x}^{2}=d$ deterministically and hence $c=d$ as well. 
  \item \textbf{Random coordinate vectors}, i.e. $x = \sqrt{d}e_{k}$ with $k$ being selected uniformly at random from $\left\{ 1,\dots,d \right\}$, i.e. $\PP(x=\sqrt{d}e_{k}) = \tfrac1d$. Since it holds $\norm{x}^{2}=d$ deterministically, it again holds that $c=d$ in this case.
  \end{enumerate}
\end{example}

\subsection{Stochastic gradient descent with adjoint sampling}
\label{sec:sgdas-convergence}

\subsubsection{Convergence of the iterates}
\label{sec:sgd}
In these paragraphs we are going to investigate convergence of the iterates of Algorithm~\ref{alg:sgd} similar to the analysis in~\cite{lorenz2018randomized}. For this to be reasonable we will need that a unique solution $\hat v$ of $Av=b$ exists, i.e. that $A^{T}A$ has full rank.


\begin{theorem}
\label{thm:taubound}
Let $Av=b$ be a consistent system with solution $\hat{v}$ and let the sequence $(v^k)_k$ be generated by Algorithm \ref{alg:sgd}, where the $x$ are sampled according to Assumption \ref{ass:distribution}.
Then it holds that 
\begin{align*}
    \EE(\norm{v^{k+1}-\hat{v}}^2) \leq \norm{v^k - \hat{v}}^2 - \tau_k \left(2-\tau_k c \norm{A}^2\right) \norm{A v^k - b}^2
\end{align*}
where the expectation is taken with respect to the random choice of $x$ is the $k$th step.  and hence is strictly decreasing and thus convergent if $0<\tau_k <2/(c\norm{A}^{2})$ for all $k$.
\end{theorem}
\begin{proof}
    Defining $g^k$ as  $g^k = \scp{Av^k-b}{Ax}x$,
    the update step can be formulated as
    $v^{k+1} = v^k - \tau_k g^k$.
    Therefore, we get
    \begin{align*}
        \norm{v^{k+1} - \hat{v}}^2 &= \norm{v^k - \tau_k g^k - \hat{v}}^2 \\
        &= \norm{v^k - \hat{v}}^2 - 2 \tau_k \scp{v^k - \hat{v}}{g^k} + \tau_k^2 \norm{g^k}^2.
    \end{align*}
    The summands can be rewritten as
    \begin{align*}
        \scp{v^k-\hat{v}}{g^k} &= \scp{Av^k - b}{Ax} \cdot \scp{v^k-\hat{v}}{x} = \scp{Av^k-b}{Ax} \cdot \scp{x}{v^k-\hat{v}} \\
        &= (Av^k - b)^T A  x x^T  (v^k -\hat{v})
    \end{align*}
    and
    \begin{align*}
        \norm{g^k}^2 &= \scp{Av^k - b}{Ax}^2 \norm{x}^2 = \scp{Av^k - b}{Ax} \scp{Ax}{Av^k -b} \norm{x}^{2} \\
        &= (Av^k - b)^T Axx^T\norm{x}^{2} A^T(Av^k -b).
    \end{align*}
    Hence it follows from Assumption~\ref{ass:distribution} that the expectation with respect to $x$ is
    \begin{align*}
    \EE(\norm{v^{k+1}-\hat{v}}^2) &= \norm{v^k - \hat{v}}^2 - 2 \tau_k \EE( \scp{v^k - \hat{v}}{g^k}) + \tau_k^2 \EE(\norm{g^k}^2) \\
    &= \norm{v^k - \hat{v}}^2 - 2 \tau_k \scp{A(v^k-\hat{v})}{Av^k-b} + \tau_k^2 c \norm{A^T(Av^k -b)}^2 \\
    &\leq \norm{v^k - \hat{v}}^2 - 2 \tau_k \norm{Av^k - b}^2 + \tau_k^2 c \norm{A}^2 \norm{Av^k-b}^2 \\
    &= \norm{v^k - \hat{v}}^2 - \tau_k \left(2-\tau_k c \norm{A}^2\right) \norm{Av^k-b}^2.
    \end{align*}
  \end{proof}
  
\begin{theorem}
\label{thm:onestep}
Let $Av = b$ be a consistent system with solution $\hat v$ and let the sequence $(v^k)_k$ be generated by Algorithm \ref{alg:sgd} where the $x$ are sampled according to Assumption \ref{ass:distribution}, the $\tau_{k} = \tau$ are constant and fulfill
\begin{align*}
0 < \tau < \frac{2}{c \norm{A}^2}.
\end{align*}
If
\begin{align}\label{eq:lambda}
\lambda \vcentcolon = \rho\left(I_d -\tau A^T A (2I-\tau c A^T A)\right) \in [0,1)
\end{align}
is fulfilled, it holds that the iterates converge linearly with
\begin{align*}
\EE(\norm{v^{k+1}-\hat{v}}^2) \leq \lambda^{k+1} \norm{v^{0}-\hat{v}}^2.
\end{align*}
\end{theorem}
\begin{proof}
As in the proof of Theorem~\ref{thm:taubound} we get
\begin{align*}
\EE(\norm{v^{k+1} - \hat{v}}^2) =& \norm{v^{k} - \hat{v}}^2 - 2 \tau \EE\left((v^{k}-\hat{v})^T A^T A x x^T (v^{k} - \hat{v}) \right) \\ &+ \tau^2 \EE\left((v^{k}-\hat{v})^T A^T A x x^T \norm{x}^{2} A^T A (v^{k} - \hat{v}) \right) \\
=& \norm{v^{k} - \hat{v}}^2 - 2 \tau \norm{A(v^{k} - \hat{v})}^2 + \tau^2 c \norm{A^TA(v^{k} - \hat{v})}^2 \\
=& \scp{v^{k}- \hat{v}}{(I- \tau A^T A (2 I - \tau c A^T A)) (v^{k} - \hat{v})} \\
\leq& \lambda \norm{v^{k} - \hat{v}}^2.
\end{align*}
Iterating this result and using the law of total expectation gives the convergence result.
\end{proof}

\begin{remark}\label{rem:opt-stepsize}
   The matrices $A^{T}A$ and $I-\tau A^T A (2I-\tau c A^T A)$ (in the argument of $\rho$ in the definition of $\lambda$) share the eigenvectors and it holds that if $\mu$ is an eigenvalue of $A^{T}A$,then $1-  \tau \mu (2 - \tau c \mu)$ is an eigenvalue of $I-\tau A^T A (2I-\tau c A^T A)$.  
  Convergence of $\EE(\norm{v^{k}-\hat{v}^{2}})$ happens when $\lambda<1$ holds and for this to hold it is necessary that $A^{T}A$ has full rank and that $\tau$ fulfills the bounds stated in Theorem~\ref{thm:onestep}.
  
  For fastest convergence, one minimizes $\lambda$ with respect to the stepsize $\tau$.
  The spectral radius $\lambda = \rho\left(I_d -\tau A^T A (2I-\tau c A^T A)\right)$ is the maximum of the values
$1-  \tau \mu (2 - \tau c \mu)$ where $\mu$ ranges over the eigenvalues of $A^{T}A$. Since this is a quadratic polynomial in $\mu$ with a positive leading term, 
the function $\Theta_{\tau}(\mu) \vcentcolon=  1-  \tau \mu (2 - \tau c \mu)$ is convex and thus, its maxima are at the endpoints of evaluation. Hence, $\lambda$ is given by
\begin{align*}
\lambda = \max\left\{\Theta_{\tau}(\lambda_{\min}(A^TA)), \Theta_{\tau}(\lambda_{\max}(A^TA))\right\}.
\end{align*}
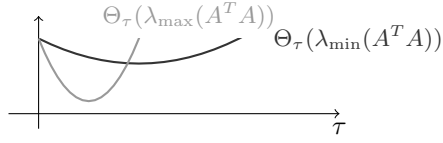
\begin{figure}[htb]
  \centering
  \begin{tikzpicture}[xscale=2]
    \draw[->] (-0.2,0) -- (2,0) node[below]{$\tau$};
    \draw[->] (0,-0.2) -- (0,1.3);

    \begin{scope}
      \clip (-0.5,-0.5) rectangle (3,1);
      \draw[domain=0:2,thick,black!80] plot (\x,{1-\x*0.5*(2-\x*0.5*3)});
      \draw[domain=0:1,thick,black!40] plot (\x,{1-\x*2.5*(2-\x*1*3)});
    \end{scope}
    \node[black!80] at (1.5,1)[right]{\small $\Theta_{\tau}(\lambda_{\min}(A^{T}A))$};
    \node[black!40] at (1,1)[above]{\small $\Theta_{\tau}(\lambda_{\max}(A^{T}A))$};
  \end{tikzpicture}
  \caption{How to find the optimal stepsize? Illustration for the derivation in Remark~\ref{rem:opt-stepsize}.}
  \label{fig:opt-stepsize}
\end{figure}
To determine the optimal stepsize $\tau$, we minimize $\lambda$ over $\tau$, i.e. we have to minimize $\max\left\{\Theta_{\tau}(\lambda_{\min}(A^TA)), \Theta_{\tau}(\lambda_{\max}(A^TA))\right\}$ over $\tau$. As can be seen in Figure~\ref{fig:opt-stepsize}, the minimum is where $\Theta_{\tau}(\lambda_{\min}(A^TA)) = \Theta_{\tau}(\lambda_{\max}(A^TA))$, and this results in
\begin{align*}
1- \tau \lambda_{\min}(A^TA) (2-\tau c \lambda_{\min}(A^TA)) &= 1- \tau \lambda_{\max}(A^TA) (2-\tau c \lambda_{\max}(A^TA)) \\
\Leftrightarrow \lambda_{\min}(A^TA) (2-\tau c \lambda_{\min}(A^TA)) &= \lambda_{\max}(A^TA) (2-\tau c \lambda_{\max}(A^TA)) 
\end{align*}
Hence, we get as optimal stepsize
\begin{align*}
  \tau &= \frac{2}{c} \frac{\lambda_{\max}(A^TA) - \lambda_{\min}(A^TA)}{\lambda_{\max}(A^TA)^2 - \lambda_{\min}(A^TA)^2} = \frac{2}{c} \frac{1}{\lambda_{\max}(A^TA) + \lambda_{\min}(A^TA)}
\end{align*}
Plugging this in we get that the optimal convergence rate is 
\begin{align*}
  \lambda = 1 - \frac4c \frac{\lambda_{\max}(A^{T}A)\lambda_{\min}(A^TA)}{(\lambda_{\max}(A^{T}A) + \lambda_{\min}(A^TA))^{2}}.
\end{align*}
With $\kappa(A) = \lambda_{\max}(A^TA)/\lambda_{\min}(A^TA)$ this is 
\begin{align*}
\lambda = 1 - \frac4c \frac{\kappa(A)}{(\kappa(A)+1)^{2}}.
\end{align*}
\end{remark}

Now we consider the inconsistent case, in which we do not assume the existence of a solution. We model this as an additive error and assume that the right hand side is $b+r$ with $b \in \rg(A)$.

\begin{theorem}\label{thm:inconsistent-estimate}
Let $\hat{v}$ fulfill $A \hat{v} = b$ and let the sequence $(v^k)_k$ be generated by Algorithm \ref{alg:sgd} with $x$ sampled according to Assumption \ref{ass:distribution}, constant stepsize $\tau$ which fulfills
\begin{align*}
0 < \tau < \frac{2}{c \norm{A}^2},
\end{align*}
and where the right hand side is changed to $b+r$ with $b \in \rg(A)$ and $r = r' + r''$ with $r' \in \rg(A)$ and $r'' \in \rg(A)^\bot$. 
Then with $\lambda$ from \eqref{eq:lambda} it holds that
\begin{align*}
\EE(\norm{v^{k+1}-\hat{v}}^2) \leq \left( \frac{1+\lambda}{2}\right)^{k+1} \norm{v^{0}-\hat{v}}^2 + \tau^2 \frac{2 \left( (1-\lambda) c + 2 \norm{I-\tau c A^T A}^2 \right)}{(1- \lambda)^2}  \norm{A^T r'}^2
\end{align*}
\end{theorem}
\begin{proof}
Following the proof of Theorem \ref{thm:onestep}, we observe
\begin{align*}
\norm{v^{k+1} - \hat{v}}^2 =& \norm{v^{k} - \tau \scp{A v^{k} - (b+r)}{Ax} x - \hat v}^2 \\
  =& \norm{v^{k} - \hat{v}}^2 - 2 \tau \left( \scp{A(v^{k}-\hat{v})}{Ax} - \scp{r}{Ax}\right) \cdot \scp{v^{k}-\hat{v}}{x} \\ &+ \tau^2 \scp{A(v^{k}-\hat{v})-r}{Ax}^2 \norm{x}^2
\end{align*}
As in Theorem \ref{thm:onestep} taking the expected value with respect to $x$ and using Young's inequality with $\varepsilon>0$ results in
\begin{align*}
  \EE(\norm{v^{k+1} - \hat{v}}^2) =& \norm{v^{k} - \hat{v}}^2 - 2 \tau \EE\left((v^{k}-\hat{v})^T A^T A x x^T (v^{k} - \hat{v}) \right) \\ &+ \tau^2 \EE\left(\left((v^{k}-\hat{v})^T A^T\right) A x x^T \norm{x}^{2} A^T \left(A (v^{k} - \hat{v})\right) \right)\\
                                   &+ 2 \tau \EE(r^T Ax x^T (v^k - \hat{v})) - 2  \tau^2 \EE(r^T Ax x^T \norm{x}^{2} A^T A(v^k-\hat{v})) \\
                                   &+ \tau^2 \EE(r^T Axx^T\norm{x}^{2} A^T r) \\
  =& \norm{v^{k} - \hat{v}}^2 - 2 \tau \norm{A(v^{k} - \hat{v})}^2 + \tau^2 c \norm{A^TA(v^{k} - \hat{v})}^2 \\
                                   &+ 2 \tau \scp{A^T r}{v^k - \hat{v}} - 2 \tau^2 c \scp{A^T r}{A^T A(v^k-\hat{v})} + \tau^2 c \norm{A^T r}^2\\
  =&  \scp{v^{k}- \hat{v}}{(I- \tau A^T A (2 I - \tau c A^T A)) (v^{k} - \hat{v})} \\ 
                                   &+ \scp{v^k-\hat{v}}{2 \tau (I -  \tau c A^T A) A^T r} + \tau^2 c \norm{A^T r}^2 \\
  \leq& \left(\lambda +  \frac{1}{2 \varepsilon}\right) \norm{v^{k} - \hat{v}}^2+ \tau^2 \underbrace{ \left(  c + 2 \varepsilon \norm{I-\tau c
     A^T A}^2 \right)}_{=: \Lambda} \norm{A^T r}^2  ,
\end{align*}
Now we choose $\varepsilon = \frac{1}{1-\lambda}$, which yields by the law of total expectation
\begin{align*}
\EE(\norm{v^{k} - \hat{v}}^2) &\leq \left(\frac{1+\lambda}{2}\right)^k \norm{v^{0} - \hat{v}}^2 + \tau^2 \sum_{j=0}^{k-1} \left(\frac{1+\lambda}{2}\right)^j \Lambda \norm{A^T r}^2 \\
&= \left(\frac{1+\lambda}{2}\right)^k \norm{v^{0} - \hat{v}}^2 + \tau^2 \Lambda \norm{A^T r} \sum_{j=0}^{k-1} \left(\frac{1+\lambda}{2}\right)^j \\
&\leq \left(\frac{1+\lambda}{2}\right)^k \norm{v^{0}-\hat{v}}^2 + \tau^2 \frac{2 \Lambda}{1- \lambda}  \norm{A^T r}^2,
\end{align*}
where the last inequality is due to the geometric series. 
Since $\rg(A)^\bot = \ker(A^T)$, we attain $A^Tr'' = 0$ and hence the result.
\end{proof}

In Theorem~\ref{thm:inconsistent-estimate} we see that we can estimate the error with two terms. The first one decays to zero exponentially, but the second is a constant offset which is due to the noise in the right hand side. This offset only depends on the part of the noise that is the range of $A$. We remark that the offset is not meaningful in the case where $\lambda=1$ occurs (which would also happen in the case of an ill-posed compact operator between Hilbert spaces---a situation which is not considered in this paper).

\subsubsection{Convergence of the residual}
\label{sec:residual}

In the results in the previous section we needed $\lambda_{\min}(A^{T}A)>0$ to get that $\lambda <1$ and hence, the results are not useful for problems where $A^{T}A$ is rank deficient, i.e. especially not in the case of underdetermined systems. In this case we do not have unique solutions of $Av=b$, but a nontrivial solution subspace (in the case of a consistent system).

We first show convergence properties of the residual. The following result holds for any linear system $Av=b$.
\begin{theorem}
\label{thm:rangeconvergence}
Let $Av=b$ be a linear system and let the sequence $(v^k)_k$ be generated by Algorithm \ref{alg:sgd}, where the $x$ are sampled according to Assumption \ref{ass:distribution} and the $\tau_k$ fulfill
\begin{align*}
    0 < \tau_k < \frac{2}{c \norm{A}^2}.
\end{align*}
Then the residual fulfills
\begin{align*}
    \EE(\norm{Av^{k+1}-b}^2) \leq \norm{Av^k - b}^2 - \tau_k \left(2-\tau_k c \norm{A}^2\right) \norm{A^T(A v^k - b)}^2
\end{align*}
where the expectation is taken with respect to the random choice of $x$ in the $k$th step and hence is strictly decreasing and thus convergent.
\end{theorem}
\begin{proof}
    Taking the expected value of
    \begin{align*}
        \norm{Av^{k+1}-b}^2 =& \norm{Av^{k} - b}^2 - 2 \tau_k \scp{Av^{k}-b}{A g^k} + \tau_k^2 \norm{Ag^k}^2 \\
        =& \norm{Av^{k} - b}^2 - 2 \tau_k (Av^k-b)^T A xx^T A^T(Av^k-b) \\
        &+ \tau_k^2 (Av^k-b)^T A x x^T A^T A x x^T A^T (Av^k -b)
    \end{align*}
    with respect to $x$ and using Assumption~\ref{ass:distribution} results in
    \begin{align*}
        \EE(\norm{Av^{k+1}-b}^2) & =\norm{Av^k-b}^2 - 2 \tau_k \norm{A^T(Av^k-b)}^2 \\
                                  &\qquad+ \tau_k^2 \EE(\norm{Axx^T A^T(Av^k-b)}^2) \\
                                  &\leq  \norm{Av^k-b}^2 - 2 \tau_k \norm{A^T(Av^k-b)}^2 \\
        &\qquad + \tau_k^2 \norm{A}^{2}\EE(\norm{xx^T A^T(Av^k-b)}^2) \\
        &\leq  \norm{Av^k-b}^2 - 2 \tau_k \norm{A^T(Av^k-b)}^2 \\ & \qquad + \tau_k^2  \norm{A}^{2}\EE((Av^{k}-b)^{T}Axx^{T}xx^{T}A^{T}(Av^k-b))\\
        &\leq  \norm{Av^k-b}^2 - 2 \tau_k \norm{A^T(Av^k-b)}^2 \\ &\qquad + \tau_k^2  \norm{A}^2c \norm{A^T(Av^k-b)}^2.
    \end{align*}
    This leads to
    \begin{align*}
        \EE(\norm{Av^{k+1}-b}^2) \leq \norm{Av^k-b}^2 - \tau_k \left(2 - \tau_k c \norm{A}^2\right) \norm{A^T(Av^k-b)}^2,
    \end{align*}
    which leads to the results.
\end{proof}

Without any assumption on the shape of $A$, rank or consistency we always get sublinear convergence of the residual of the normal equations.
\begin{theorem}\label{thm:sublinear-ls-residual}
    Let the sequence $(v^k)_k$ be generated by Algorithm \ref{alg:sgd} with constant step size
    \begin{align*}
        \tau_k = \tau \vcentcolon = \frac{1}{c \norm{A}^2},
    \end{align*} 
    where the $x$ are sampled according to Assumption \ref{ass:distribution}. Then it holds that
    \begin{align*}
        \min_{0 \leq k \leq N-1} \EE(\norm{A^T(Av^k-b)}^2) \leq \frac{c \norm{A}^2 \norm{b}^2}{N}.
    \end{align*}
\end{theorem}
\begin{proof}
    Under the given conditions, Theorem \ref{thm:rangeconvergence} and the law of total expectation provide
    \begin{align*}
        \EE(\norm{A^T(Av^k-b)}^2) \leq c \norm{A}^2 \left(\EE(\norm{Av^k-b}^2) - \EE(\norm{Av^{k+1}-b}^2)\right).
    \end{align*}
    Hence summing up yields to
    \begin{align*}
        N \cdot \min_{0 \leq t \leq N-1} \EE(\norm{A^T(Av^k-b)}^2) &\leq \sum_{k=0}^{N-1} \EE(\norm{A^T(Av^k-b)}^2) \\
        &\leq c \norm{A}^2 \left( \norm{Av^0-b}^2 - \EE(\norm{Av^{N}-b}^2)\right)
    \end{align*}
    Since $v^0 = 0$, this becomes
    \begin{align*}
    N \cdot \min_{0 \leq t \leq N-1} \EE(\norm{A^T(Av^k-b)}^2) &\leq c \norm{A}^2 \norm{b}^2
    \end{align*}
    and thus the result is proven.
\end{proof}

The following Theorem shows expected linear convergence of the residual for the iterates generated by Algorithm \ref{alg:sgd}.
\begin{theorem}
\label{thm:linearconvergenceresidual}
Let $Av = b$ be a linear system and let the sequence $(v^k)_k$ be generated by Algorithm \ref{alg:sgd} with constant stepsize
\begin{align*}
0 < \tau < \frac{2}{c \norm{A}^2},
\end{align*}
where the $x$ are sampled according to Assumption \ref{ass:distribution}. 

Then it holds for the expectation taken with respect to the random choice of $x$ in the $k$th step that
\begin{align*}
\EE(\norm{Av^{k+1}-b}^2) \leq \beta \norm{Av^{k}-b}^2.
\end{align*}
with
\begin{align*}
\beta \vcentcolon = 1-\tau\sigma_{\min}(A)^{2} (2 - \tau c \norm{A}^2).
\end{align*}
If $\sigma_{\min}(A)>0$ we have $0<\beta<1$ 
and hence we have linear convergence
\begin{align*}
\EE(\norm{Av^{k+1}-b}^2) \leq \beta^{k+1} \norm{Av^{0}-b}^2.
\end{align*}
\end{theorem}
\begin{proof}
    Similar to the proof of Theorem \ref{thm:onestep} we calculate
    \begin{align*}
\norm{Av^{k+1} - b}^2 =& \norm{Av^{k} - \tau \scp{A v^{k} - b}{Ax} Ax - b}^2  \\
=& \norm{Av^{k} - b}^2 - 2 \tau \scp{Av^{k}-b}{Ax} \cdot \scp{Av^{k}-b}{Ax} \\ &+ \tau^2 \scp{Av^{k}-b}{Ax}^2 \norm{Ax}^2 \\
      \leq& \norm{Av^{k} - b}^2 - 2 \tau \scp{Av^{k}-b}{Ax}^2 + \tau^2 \norm{A}^2 \scp{Av^{k}-b}{Ax}^2\norm{x}^{2}
\end{align*}
Considering $\EE(xx^T) = I$, taking the expectation with respect to $x$ leads to
\begin{align*}
  \EE(\norm{Av^{k+1} - b}^2) & \leq \norm{Av^{k} - b}^2 - 2 \tau \EE(\scp{Av^{k}-b}{Ax}^2) \\ &\quad+ \tau^2 \norm{A}^2 \EE(\scp{Av^{k}-b}{Ax}^2\norm{x}^{2})\\
  & = \norm{Av^{k} - b}^2 - 2 \tau \norm{A^{T}(Av^{k}-b)}^2 + \tau^2 \norm{A}^2 c \norm{A^{T}(Av^k-b)}^{2}\\
                             & = \norm{Av^{k} - b}^2 - \tau (2- \tau c \norm{A}^2 ) \norm{A^T (Av^k - b)}^2 \\
  & \leq \norm{Av^{k}-b}^{2} - \tau(2-\tau c\norm{A}^{2}) \sigma_{\min}(A)^{2}\norm{Av^{k}-b}^{2}\\
 &= \beta  \norm{Av^{k} - b}^2.
\end{align*}
This proves the first statement. By induction and the law of total expectation, the second result follows. 
\end{proof}

Minimizing $\beta$ over $\tau$ leads to the optimal stepsize $\tau = 1/(c\norm{A}^{2})$ and
\begin{align*}
\beta = 1 - \tfrac{\sigma_{\min}(A)^{2}}{c\norm{A}^{2}}.
\end{align*}
This shows:

\begin{corollary}
\label{cor:linearconvergenceresidual}
Let $Av = b$ be a linear system with $\sigma_{\min}(A)>0$ and let the sequence $(v^k)_k$ be generated by Algorithm \ref{alg:sgd} with
\begin{align*}
\tau = \frac{1}{c \norm{A}^2},
\end{align*}
where the $x$ are sampled according to Assumption \ref{ass:distribution}. 

Then it holds that
\begin{align*}
\EE(\norm{Av^{k}-b}^2) \leq \left(1-\frac{\sigma_{\min}(A)^{2}}{c \norm{A}^2}\right)^k \norm{Av^{0}-b}^2
\end{align*}
\end{corollary}

In conclusion, we have shown that the stochastic gradient method with adjoint sampling leads to a linear converge rate in terms of distance to the solution (if it is unique) and the residual (if the system is consistent). For systems which are not full rank we still get sublinear convergence of the least squares residual. The standard stochastic gradient method for the least squares functional where the objective is written as 
$\sum_{i=1}^{m} (a_{i}^{T}v - b_{i})^{2}$
with rows $a_{i}^{T}$ of $A$, leads (after taking the steepest descent stepsize) to the randomized Kaczmarz method. For this one obtains the convergence of the iterates (and residuals) with linear rate $\left( 1 - \frac{\sigma_{\min}(A)^{2}}{\norm{A}_{F}^{2}} \right)$ (cf.~\cite{strohmer2009kaczmarz}). Since $\norm{A}_{F}^{2}\leq d\norm{A}^{2}$ and $c\geq d$, the rate in Corollary~\ref{cor:linearconvergenceresidual} is slightly worse than the one for the randomized Kaczmarz method. However, the Kaczmarz method only needs access to a single row per iterations (something we do not have access to in the situation of the this paper), while SGDAS needs two applications of the linear map per iteration.

All our results have at least two drawbacks:
\begin{enumerate}
\item The convergence is slow due to the factor $c$ which is, in all cases of our random sampling, larger than the input dimension $d$. 
\item The stepsize conditions depend on the operator norm $\norm{A}$, which is usually not known under our assumptions. Both problems can be circumvented by random descent as we show in the next section.
\end{enumerate}

\subsection{Random descent}
\label{sec:rd-convergence}

Now we turn to the analysis of the random descent algorithm (Algorithm~\ref{alg:rd}). While this algorithm just takes steps in random directions, we will be able to show linear convergence towards minimizers by using optimal stepsizes. Hence, as a first step, we calculate the optimal stepsize for minimizing the least squares functional in some search direction. 

\begin{lemma}
\label{lem:rdtau}
    The minimum of $\tau\mapsto \tfrac{1}{2} \norm{A(v^k+\tau x) - b}^2$
    is attained at
    \begin{align*}
        \tau_k = \left\{\begin{array}{cl} - \frac{\scp{Av^k - b}{Ax}}{\norm{Ax}^2} & Ax \neq 0, \\ 0 & Ax = 0. \end{array}\right.
    \end{align*}
\end{lemma}
\begin{proof}
If $Ax = 0$, the objective does not depend on $\tau$ at all.
Otherwise the objective is convex and the minimizer is provided by the solution of
\begin{align*}
    0 &= \frac{\partial}{\partial \tau} \frac{1}{2} \norm{A(v^k+\tau x) - b}^2 =  \scp{A^T(Av^k + \tau A x - b)}{x} =  \scp{Av^k-b}{Ax} + \tau \norm{Ax}^2,
\end{align*}
i.e. $\tau = - \frac{\scp{Av^k-b}{Ax}}{\norm{Ax}^2}$.
\end{proof}

Note that this stepsize may be positive or negative. Most notably, we can evaluate this stepsize under our restrictive assumptions stated in the introduction: We only need to evaluate $A$ two times (once at the direction $x$ and once at the current iterate). Moreover, no knowledge of the operator norm of $A$ is needed to use this stepsize in practice. Since this step is optimal in direction $x$, we immediately get linear convergence of random descent (Algorithm~\ref{alg:rd}) with this stepsize  by comparison with our results for the stochastic gradient descent (Algorithm~\ref{alg:sgd}). If $v^{k+1} = v^{k} - \tau_{k}x$  from the random descent with optimal stepsize and $\tilde v^{k+1} = v^{k} - \tau \scp{Av^{k}-b}{Ax}x$ with some $\tau$ between zero and $2/(c\norm{A}^{2})$ we always have 
\begin{align*}
\norm{Av^{k+1}-b}^{2}\leq \norm{A\tilde v^{k+1}-b}^{2},
\end{align*}
i.e. random descent does decrease the residual faster than stochastic gradient descent with adjoint sampling. Thus, we can transfer some results from the previous section to random descent:

\begin{theorem}
\label{thm:convergence-random-descent}
Let $Av=b$ be a linear system and let the sequence $(v^k)_k$ be generated by Algorithm \ref{alg:rd}, where the $x$ are sampled according to Assumption \ref{ass:distribution} and the $\tau_k$ are given by
\begin{align*}
\tau_k = - \tfrac{\scp{Av^k - b}{Ax}}{\norm{Ax}^2}
\end{align*}
Then we have 
\begin{align*}
  \min_{0 \leq k \leq N-1} \EE(\norm{A^T(Av^k-b)}^2) \leq \frac{c \norm{A}^2 \norm{b}^2}{N}.
\end{align*}
Moreover, the sequence $\left(\EE(\norm{Av^k-b})\right)_k$ of the expected residuals is decreasing.
If furthermore $\lambda_{\min}(AA^{T})>0$, then the sequence of expected residuals
$\left(\EE(\norm{Av^k-b}^2)\right)_k$
does converge at least linearly, i.e. we have
\begin{align*}
  \EE(\norm{Av^k-b}^2) \leq \left(1- \tfrac{\lambda_{\min}(AA^{T})}{c \norm{A}^2} \right)^{k} \norm{Av^0 - b}^{2}.
\end{align*}
\end{theorem}
\begin{proof}
    The first result follows directly from Lemma \ref{lem:rdtau}, while the second is an immediate result of the comparison with Theorem \ref{thm:linearconvergenceresidual}, Corollary \ref{cor:linearconvergenceresidual}, in which linear convergence is proven for non-optimal fixed step sizes.
\end{proof}

We would expect a better convergence rate due to the optimal stepsize. It turns out that the refined convergence analysis differs for different sampling schemes for the random vectors $x$.  The central object is the matrix
\begin{align}
  M := \EE \left( \tfrac{xx^{T}}{\norm{Ax}^{2}} \right)\in\RR^{d\times d}\label{eq:def-M}
\end{align}
where we have to assume that the expectation exists (see Remark~\ref{rem:M-matrix} below).
One sees that $M$ only depends on the marginal distribution on the unit sphere, so all results that we obtain for the standard normal distribution also apply to the uniform distribution on the sphere. The next theorem shows the role of the matrix $M$. After the theorem we discuss the usefulness of the different estimates in different cases.
\begin{theorem}\label{thm:convergence-random-descent-refined}
  Let $v^{k}$ be generated by Algorithm~\ref{alg:rd} and assume that $M$
  exists. Then it holds
  \begin{enumerate}
  \item For the expectation with respect to the choice of $x$ in the $k$th step we have
    \begin{align*}
      \EE(\norm{Av^{k+1}-b}^2)  = \norm{Av^{k}-b}^{2} - \scp{A^{T}(Av^{k}-b)}{M(A^{T}(Av^{k}-b)}.
    \end{align*}
  \item If $\lambda_{\min}(M)>0$ it holds that 
\begin{align}\label{eq:conv-ls-residual}
\min_{0\leq k\leq N-1}\norm{A^{T}(Av^{k}-b)}^{2} \leq \tfrac{\norm{b}^{2}}{\lambda_{\min}(M)N}.
\end{align}
\item It holds that 
\begin{align}\label{eq:conv-residual-rd-d-leq-m}
\EE(\norm{Av^{k+1}-b}^{2})\leq (1-\lambda_{\min}(M)\sigma_{\min}(A)^{2})\EE(\norm{Av^{k}-b}^{2}).
\end{align}
and
\begin{align}\label{eq:conv-residual-rd-m-leq-d}
\EE(\norm{Av^{k+1}-b}^{2})\leq (1-\lambda_{\min}(AMA^{T}))\EE(\norm{Av^{k}-b}^{2}).
\end{align}
  \end{enumerate}
\end{theorem}
\begin{proof}
 For the first claim we calculate 
\begin{align*}
  \norm{Av^{k+1}-b}^2  & = \norm{A(v^{k}-\tfrac{\scp{Av^{k}-b}{Ax}}{\norm{Ax}^{2}}x)-b}^{2}\nonumber \\
                            & = \norm{Av^{k}-b}^{2} - 2\tfrac{\scp{Av^{k}-b}{Ax}^{2}}{\norm{Ax}^{2}} + \tfrac{\scp{Av^{k}-b}{Ax}^{2}\norm{Ax}^{2}}{\norm{Ax}^{4}}\nonumber\\
                       & = \norm{Av^{k}-b}^{2} - \tfrac{\scp{Av^{k}-b}{Ax}^{2}}{\norm{Ax}^{2}\nonumber}\\
  & = \norm{Av^{k}-b}^{2} - \scp{A^{T}(Av^{k}-b)}{\tfrac{xx^{T}}{\norm{Ax}^{2}}A^{T}(Av^{k}-b)},\label{eq:descent-residual}
\end{align*}
and taking the expectation over the random choice of $x$ proves the claim.

For the second claim we use the law of total expectation and estimate 
\begin{align*}
\EE(\norm{Av^{k+1}-b}^2) \leq \EE(\norm{Av^{k}-b}^{2}) - \lambda_{\min}(M)\norm{A^{T}(Av^{k}-b)}^{2}
\end{align*}
and rearranging, summing up and estimating by the minimum gives the claim (similar as in the proof of Theorem~\ref{thm:sublinear-ls-residual}).

For the third claim we further estimate $\norm{A^{T}(Av^{k}-b)}\geq \sigma_{\min}(A)\norm{Av^{k}-b}$ and get 
\begin{align*}
  \EE(\norm{Av^{k+1}-b}^{2}) &  \leq \norm{Av^{k}-b}^2 -  \lambda_{\min}(M)\sigma_{\min}(A)^{2}\norm{Av^{k}-b}^{2}.
\end{align*}
The second inequality in 3. follows directly from 1.
\end{proof}

\begin{remark}\label{rem:M-matrix}
  The two estimates in~\eqref{eq:conv-residual-rd-d-leq-m} and \eqref{eq:conv-residual-rd-m-leq-d} in Theorem~\ref{thm:convergence-random-descent-refined} 3. are relevant in different circumstances: Estimate~\eqref{eq:conv-residual-rd-m-leq-d} with the contraction factor $(1-\lambda_{\min}(M)\sigma_{\min}(A)^{2})$ is relevant if $\lambda_{\min}(M)>0$. Whether this is true depends on the dimensions of $A$, the distribution of the directions $x$, and the matrix $A$ (see below for further details). For overdetermined systems $A$, i.e. for $m>d$ and $A$ with rank $d$ this is true for all the isotropic sampling schemes we considered here.

The estimate~\eqref{eq:conv-residual-rd-m-leq-d} with contraction factor $(1-\lambda_{\min}(AMA^{T}))$ is only relevant if the minimal eigenvalue is positive. This cannot happen when $m>d$ since $AMA^{T}$ is $m\times m$ with rank at most $d$. In the case $m<d$ the matrix $A$ has a nontrivial kernel and hence, the expectation $\EE(\tfrac{xx^{T}}{\norm{Ax}^{2}})$ may not exist (in fact, it does not exist for $x\sim\cN(0,I_{d})$ and $x\sim\Unif(\sqrt{d}S^{d-1})$). However, for the discrete distributions (random coordinate vectors and Rademacher vectors), this depends on whether $\PP(x\in\ker(A))>0$ or not. If this is not the case, the expectation exists and is finite. Moreover, in this case we have $\lambda_{\min}(AMA^{T})>0$ for $d>m$ if $\rank(A)=m$.
\end{remark}

The convergence speed is governed by spectral properties of $M$. A simple and crude estimate shows the following bounds for the eigenvalues of $M$ which holds for all isotropic random vectors:

\begin{proposition}[Estimates for general isotropic random vectors]\label{prop:expectation-M-AMA}
  Let $d\leq m$, $A\in\RR^{m\times d}$ and $x\in\RR^{d}$ be an isotropic random vector. Then it holds for $M = \EE(\tfrac{xx^T}{\norm{Ax}^2})$ that 
\begin{align*}
\frac1{d\sigma_{\max}(A)^{2}}I_{d} \mleq M \mleq \frac1{d\sigma_{\min}(A)^{2}}I_{d}.
\end{align*}
Consequently it holds that 
\begin{align*}
\tfrac1{d\sigma_{\max}(A)^{2}}AA^{T} \mleq\EE(\tfrac{Axx^TA^{T}}{\norm{Ax}^2})  \mleq \tfrac1{d\sigma_{\min}(A)^{2}}AA^{T}\mleq \tfrac1d \kappa^{2}I_{d}.
\end{align*}
with $\kappa = \sigma_{\max}(A)^{2}/\sigma_{\min}(A)^{2}$ being the condition number of $A$.
\end{proposition}
\begin{proof}
  Since $m\geq d$ we have $\sigma_{\min}(A)^{2}\norm{x}^{2}\leq \norm{Ax}^{2}\leq \sigma_{\max}(A)^{2}\norm{x}^{2}$ and hence 
\begin{align*}
 \tfrac{xx^T}{\sigma_{\max}(A)^{2}\norm{x}^{2}} \leq \tfrac{xx^{T}}{\norm{Ax}^{2}} \leq \tfrac{xx^{T}}{\sigma_{\min}(A)^{2}\norm{x}^{2}}.
\end{align*}
Hence, we only have to compute $\EE(\tfrac{xx^T}{\norm{x}^{2}})$ and for isotropic random variables it holds that $\EE(\tfrac{xx^T}{\norm{x}^{2}}) = \tfrac1d I_{d}$ which proves the claim.
\end{proof}

Unfortunately, these bounds do not imply any faster convergence for RD in comparison with SGDAS (compare the statements from Theorem~\ref{thm:sublinear-ls-residual} and Corollary~\ref{cor:linearconvergenceresidual} with Theorem~\ref{thm:convergence-random-descent-refined} and Proposition~\ref{prop:expectation-M-AMA}, respectively). However, numerical experiments indicate that the smallest and largest eigenvalues of $M$ are usually closer together than the simple bounds from Proposition~\ref{prop:expectation-M-AMA} suggest.
The following lemma shows better bounds for the standard normal distribution (and, by the above remark, the uniform distribution on the sphere):

\begin{lemma}[Estimates for the normal distribution]\label{lem:ev-estimates-M-normal-dist}
  Let $x\sim\cN(0,I_{d})$ or $x\sim\Unif(\sqrt{d}S^{d-1}) $and $m>d>2$. Then it holds that the matrices $A^{T}A$ and $M = \EE \left( \tfrac{xx^{T}}{\norm{Ax}^{2}} \right)$ can be diagonalized simultaneously. More precisely, if $A^{T}A = U\diag(\lambda_{i})U^{T}$ with orthonormal $U$ with columns $u_{i}$, $i=1,\dots,d$, then it is $M = U\diag(\mu_{i})U^{T}$ and the eigenvalues $\mu_{i}$ of $M$ fulfill
  \begin{align*}
    \frac{1}{2(\lambda_{1} + \cdots + \lambda_{d})}\frac{\Gamma(\tfrac{d}2)}{\Gamma(\tfrac{d+1}{2})}\leq \mu_i &= \EE \left( \frac{\scp{u_{i}}{x}^{2}}{\sum_{j=1}^d\lambda_{j}\scp{u_{j}}{x}^{2}} \right) \\
    & \leq \frac1{d(\lambda_{1}\cdot\cdots\cdot\lambda_{d})^{1/d}}\frac{1}{\pi^{d/2}}\Gamma(\tfrac32-\tfrac1d)\Gamma(\tfrac12-\tfrac1d)^{d-1}.
\end{align*}
\end{lemma}
\begin{proof}
  For $i,k\in \left\{ 1,\dots,d \right\}$ consider
  \begin{align*}
    (U^{T}MU)_{i,k}  = u_{i}^{T}Mu_{k} = \EE \left( \tfrac{\scp{u_{i}}{x}\scp{u_{k}}{x}}{\sum\limits_{j=1}^d\lambda_{j}\scp{u_{j}}{x}^{2}} \right).
  \end{align*}
  The random variables $\scp{u_{i}}{x}$ and $\scp{u_{k}}{x}$ are uncorrelated since $\EE(\scp{u_{i}}{x}\scp{u_{k}}{x}) = u_{i}^{T}\EE(xx^{T})u_{k} = \scp{u_{i}}{u_{k}}$ and have zero mean. Since the denominator $\sum_{j=1}^d\lambda_{j}\scp{u_{j}}{x}^{2}$ on contains squares of inner products, it is independent of the signs of $\scp{u_{i}}{x}$. This shows that $(U^{T}MU)_{i,k} = 0$ for $i\neq k$. The formula for $\mu_{i}$ follows from the case $i=k$.

  For the upper estimate we denote $z_{i} =\scp{u_{i}}{x}$ and note that $z_{i}\sim\cN(0,1)$. We use the inequality for the arithmetic and geometric mean
  \begin{align*}
    \lambda_{1}z_{1}^{2} + \cdots + \lambda_{d}z_{d}^{2} \geq d \left( \lambda_{1}z_{1}^{2} \cdot \cdots \cdot \lambda_{d}z_{d}^{2} \right)^{1/d}.
  \end{align*}
  To estimate $\mu_{i}$ we take, without loss of generality, $i=1$ (since in the following derivation we do not assume that the values $\lambda_{k}$ are ordered decreasingly) and have 
\begin{align*}
  \mu_{1} &= \tfrac1{(2\pi)^{d/2}}\int\limits_{\RR^{d}}^{}\frac{z_{1}^{2}}{\lambda_{1}z_{1}^{2} + \cdots + \lambda_{d}z_{d}^{2}}e^{-\norm{z}^{2}/2}\dd z\\
  & \leq \frac1{d(\lambda_{1}\cdot\cdots\cdot\lambda_{d})^{1/d}} \frac{1}{(2\pi)^{d/2}}\int_{\RR} (z_{1}^{2})^{1-1/d}e^{-z_{1}^{2}/2}\dd z_{1} \prod\limits_{j=2}^{d}\int_{\RR}(z_{j}^{2})^{-1/d}e^{-z_{j}^{2}/2}\dd z_{j}.
\end{align*}
We use the identity $\int_{\RR}(z^{2})^{a}e^{-z^{2}/2}\dd z = 2^{a+\tfrac12}\Gamma(a+\tfrac12)$ (valid for $a>-\tfrac12$) and obtain 
\begin{align*}
  \mu_{1}& \leq \frac1{d(\lambda_{1}\cdot\cdots\cdot\lambda_{d})^{1/d}} \frac{1}{(2\pi)^{d/2}}2^{\tfrac32-\tfrac1d}\Gamma(\tfrac32-\tfrac1d)\left( 2^{\tfrac12-\tfrac1d}\Gamma(\tfrac12-\tfrac1d) \right)^{d-1} \\
  & = \frac1{d(\lambda_{1}\cdot\cdots\cdot\lambda_{d})^{1/d}}\frac{1}{\pi^{d/2}}\Gamma(\tfrac32-\tfrac1d)\Gamma(\tfrac12-\tfrac1d)^{d-1}.
\end{align*}

Now we turn to the lower estimate. We write the integral in $d$-dimensional spherical coordinates as 
\begin{align*}
  \mu_{1}& = \tfrac1{(2\pi)^{d/2}}\int\limits_{\RR^{d}}^{}\frac{z_{1}^{2}}{\lambda_{1}z_{1}^{2} + \cdots + \lambda_{d}z_{d}^{2}}e^{-\norm{z}^{2}/2}\dd z\\
  & = \tfrac1{(2\pi)^{d/2}}\!\! \int\!\!\tfrac{\cos(\phi_{1})^{2}e^{-r^{2}/2}r^{d-1}\sin(\phi_{1})^{d-1}\sin(\phi_{2})^{d-2}\cdots\sin(\phi_{d-2})}{\lambda_{1}\cos(\phi_{1})^{2} + \lambda_{2}\sin(\phi_{1})^{2}\cos(\phi_{2})^{2} + \cdots + \lambda_{d}\sin(\phi_{1})^{2}\cdots\sin(\phi_{d-1})^{2}}\!\!\dd\! r\!\dd\! \phi_{1}\!\cdots\!\dd\! \phi_{d-1}.
\end{align*}
where the integral in the second line is taken over $(r,\phi_{1},\dots,\phi_{d-1}) \in \RR\times [0,\pi] \times\cdots\times [0,\pi]\times[0,2\pi]$.
We estimate all trigonometric functions in the denominator by $1$ and get 
\begin{align*}
  \mu_{1}&\geq \tfrac1{(2\pi)^{d/2}\sum\limits_{j=1}^{d}\lambda_{j}} \int\limits_0^{\infty}r^{d-1}e^{-r^{2}/2}\dd r \int\limits_0^{\pi}\cos(\phi_{1})^{2}\sin(\phi_{1})^{d-2}\dd \phi_{1} \int\limits_0^{\pi}\sin(\phi_{2})^{d-3}\dd\phi_{2}\cdots\\
  & \qquad \cdots\int\limits_0^{\pi}\sin(\phi_{d-2})\dd\phi_{d-2} \int\limits_0^{2\pi}\dd\phi_{d-1}.
\end{align*}
We use the identities 
\begin{align*}
  \int\limits_0^{\infty}r^{d-1}e^{-r^2/2}\dd r & = 2^{(d-2)/2}\Gamma(\tfrac{d}2),&
  \int\limits_0^{\pi} \sin(\phi)^{m}\dd\phi & = \frac{\Gamma(\tfrac{m+1}{2})\Gamma(\tfrac12)}{\Gamma(\tfrac{m+2}{2})},\\
  \int\limits_0^{\pi} \sin(\phi)^{m}\cos(\phi)^{2}\dd \phi & = \frac{\Gamma(\tfrac{m+1}{2})\Gamma(\tfrac32)}{\Gamma(\tfrac{m+3}{2})}
\end{align*}
and get (with $\Gamma(\tfrac12) = \sqrt{\pi}$ and $\Gamma(\tfrac32) = \sqrt{\pi}/2$)
\begin{align*}
  \mu_{1}& \geq \tfrac1{(2\pi)^{d/2}\sum\limits_{j=1}^{d}\lambda_{j}} 2^{(d-2)/2}\Gamma(\tfrac{d}2) \tfrac{\Gamma(\tfrac{d-1}2)\Gamma(\tfrac32)}{\Gamma(\tfrac{d+1}2)}\cdot\tfrac{\Gamma(\tfrac{d-2}2)\Gamma(\tfrac12)}{\Gamma(\tfrac{d-1}2)}\cdot\cdots\cdot\tfrac{\Gamma(\tfrac32)\Gamma(\tfrac12)}{\Gamma(2)}\tfrac{\Gamma(1)\Gamma(\tfrac12)}{\Gamma(\tfrac32)}2\pi\\
  & = \frac{2^{(d-2)/2}\Gamma(\tfrac{d}2)}{(2\pi)^{d/2}\sum_{j=1}^{d}\lambda_{j}}\frac{\Gamma(\tfrac32)\Gamma(\tfrac12)^{d-3}}{\Gamma(\tfrac{d+1}{2})}2\pi = \frac{1}{2(\lambda_{1} + \cdots + \lambda_{d})}\frac{\Gamma(\tfrac{d}2)}{\Gamma(\tfrac{d+1}{2})}.
\end{align*}
\end{proof}

\begin{remark}[Approximate bounds for the normal distribution, reduction of the dependence on the dimension]
  The lower and upper estimates in Lemma~\ref{lem:ev-estimates-M-normal-dist} are difficult to interpret. To make them easier we note that for large $d$
\begin{align*}
  \frac{\Gamma(\tfrac{d}2)}{\Gamma(\tfrac{d+1}{2})} & \approx \frac{1}{\sqrt{2d}}, & 
                                                                      \Gamma(\tfrac32 - \tfrac1d) & \approx \Gamma(\tfrac32) = \frac{\sqrt{\pi}}{2}, &
                                                                                                                                     \Gamma(\tfrac12 - \tfrac1d) \approx \Gamma(\tfrac12) = \sqrt{\pi}
\end{align*}
and thus $\Gamma(\tfrac32-\tfrac1d)\Gamma(\tfrac12-\tfrac1d)^{d-1}\approx \tfrac{\pi^{d/2}}2$.
With these approximation we have approximate bounds
\begin{align}\label{eq:approx-bounds-M-normal}
\frac1{2\sqrt{2d}\sum_{j=1}^{d}\lambda_{j}} I_{d}\precapprox M \precapprox \frac{1}{2d(\lambda_{1}\cdot\cdots\cdot\lambda_{d})^{1/d}}I_{d},
\end{align}
where $\lambda_{i} = \sigma_{i}(A)^{2}$ are the squares of the singular values of $A$. Especially, we get the bound $\lambda_{\min}(M)\geq (\sqrt{8d}\sum_{j} \lambda_{j})^{-1}$ and since $\norm{A}_{F}^{2} = \sum_{j} \lambda_{j}$ we get from equation~\eqref{eq:conv-residual-rd-d-leq-m}
\begin{align*}
\EE(\norm{Av^{k+1}-b}^2) \leq \left( 1- \tfrac{\sigma_{\min}(A)^{2}}{\sqrt{8d}\norm{A}_{F}^{2}} \right)\EE(\norm{Av^k-b}^{2})
\end{align*}
Likewise we get from~\eqref{eq:conv-ls-residual} that  
\begin{align*}
\min_{0\leq k\leq N-1}\norm{A^{T}(Av^{k}-b)}^{2} \leq \tfrac{\sqrt{8d}\norm{A}_{F}^2\norm{b}^{2}}{N}.
\end{align*}
\end{remark}

We remark that the estimates on $M$ are loose and that in practice one usually observes faster convergence. We will do a refined analysis in the next section and show how the iterates ``converge along singular vectors''. 

At the end of this section we have a closer look at the choice of random coordinate vectors.

\begin{lemma}[Estimates for random coordinate vectors]\label{lem:eigenvalue-estimates-M-coordinate}
  Let $x$ be a random coordinate vector, i.e. $x = \sqrt{d}e_{k}$ with $k$ being selected uniformly at random from $\left\{ 1,\dots,d \right\}$. Then it holds that 
\begin{align*}
M = \EE \left( \tfrac{xx^{T}}{\norm{Ax}^{2}} \right) =  \tfrac1d \diag(\norm{a_{1}}^{-2},\dots,\norm{a_{d}}^{-2})
\end{align*}
and consequently
\begin{align*}
\frac{1}{d\max\limits_{j=1\dots,d}\norm{a_{j}}^{2}}I_{d}\mleq M \mleq \frac{1}{d\min\limits_{j=1,\dots,d}\norm{a_{j}}^{2}}.
\end{align*}
\end{lemma}
\begin{proof}
  We denote the columns of $A$ by $a_{i}$, $i=1,\dots,d$, and the $\norm{Ax}_{2} = \sqrt{d}\norm{a_{k}}_{2}$ with probability $1/d$ and it directly follows
\begin{align*}
\EE \left( \tfrac{xx^{T}}{\norm{Ax}^{2}} \right) = \tfrac1d \sum\limits_{k=1}^d\tfrac{e_{k}e_{k}^{T}}{\norm{Ae_{k}}^{2}} = \tfrac1d  \diag(\norm{a_{1}}^{-2},\dots,\norm{a_{d}}^{-2}).
\end{align*}
\end{proof}

We can also consider sampling of $x$ from non-isotropic distributions, which brings an additional degree of freedom to the setup of the algorithm. However, it is in general hard to come up with sampling schemes that provably improve the convergence rate and are simple to implement in practice.
\begin{example}
 In the case of random coordinate vectors we can consider sampling $x = \sqrt{d}e_{k}$ with probability $p_{k}$ and get the expectation 
\begin{align*}
M = \EE(\tfrac{xx^{T}}{\norm{Ax}^{2}}) =\diag(p_{1}\norm{a_{1}}^{-2},\dots,p_{d}\norm{a_{d}}^{-2}).
\end{align*}
For the special choice $p_{k} = \norm{a_{k}}^{2}/\norm{A}_{F}^{2}$ (know from the randomized Kaczmarz method~\cite{strohmer2009kaczmarz}) we obtain 
\begin{align*}
M = \tfrac1{\norm{A}_{F}^{2}}I_{d}
\end{align*}
which leads to the rate 
\begin{align*}
\EE(\norm{Av^{k+1}-b}^2) \leq \left( 1- \tfrac{\sigma_{\min}(A)^{2}}{\norm{A}_{F}^{2}} \right)\EE(\norm{Av^k-b}^{2})
\end{align*}
which is known from~\cite{leventhal2010randomized} (see also~\cite{gower2015randomized}).
\end{example}

\section{Ill posed problems and convergence along singular vectors}
\label{sec:illposed}

In this section we will analyze the behavior of the quantities $\scp{v^{k}-\hat v}{u_{i}}$ for iterates $v^{k}$ and right singular vectors $u_{i}$ of $A$. These quantities have been studied in \cite{steinerberger2021randomized} (and previous results in a similar direction can be found in~\cite{jiao2017preasymptotic}) for the randomized Kaczmarz iteration and it has been shown that their expectation vanishes at different rates. We will analyze the noisy case, i.e. the case where $A\hat v = b$  but the iteration is run with the right hand side $b+r$. As a baseline for comparison, we first analyze the iterates of the Landweber iteration.
\begin{theorem}\label{thm:conv-singvec-lw}
  Let $A\in\RR^{d\times m}$ with right and left singular vectors $(u_{i})$ and $(w_{i})$, respectively, and singular values $\sigma_{i}$. Moreover let $b\in\RR^{m}$ and $\hat{v}\in\RR^{d}$ with $A\hat{v}= b$ and $\tilde b =b+r$, $v^{0}\in\RR^{d}$ and $v^{k}$ defined by 
\begin{align*}
v^{k+1} = v^{k} - \omega A^{T}(Av^{k}-\tilde b)
\end{align*}
for some stepsize $\omega>0$.
Then it holds that 
\begin{align*}
\scp{v^{k}-\hat{v}}{u_{i}} = (1 - \omega\sigma_{i}^{2})^{k}\scp{v^{0}-\hat{v}}{u^{i}} + \tfrac{1-(1-\omega\sigma_{i}^{2})^{k}}{\sigma_{i}}\scp{r}{w_{i}} .
\end{align*}
\end{theorem}
\begin{proof}
We calculate  
\begin{align*}
  \scp{v^{k+1}-\hat{v}}{u_{i}} & = \scp{v^{k}-\hat{v}}{u_{i}} - \omega\scp{Av^{k}-b-r}{Au_{i}}\\
                               & =  \scp{v^{k}-\hat{v}}{u_{i}} - \omega\scp{A(v^{k}-\hat{v})}{Au_{i}} + \omega\scp{r}{Au_{i}}\\
                               & =  \scp{v^{k}-\hat{v}}{u_{i}} - \omega\scp{v^{k}-\hat{v}}{A^{T}Au_{i}} + \omega\sigma_{i}\scp{r}{w_{i}}\\
  & = (1-\omega\sigma_{i}^{2})\scp{v^{k}-\hat{v}}{u_i} + \omega\sigma_{i}\scp{r}{w_{i}}.
\end{align*}
Recursively, this gives 
\begin{align*}
  \scp{v^{k}-\hat{v}}{u_{i}} &= (1-\omega\sigma_{i}^{2})^{k}\scp{v^{0}-\hat{v}}{u_{i}} + \sum\limits_{l=0}^{k-1}(1-\omega\sigma_{i}^{2})^{l}\omega\sigma_{i}\scp{r}{w_{i}}\\
                             &  = (1-\omega\sigma_{i}^{2})^{k}\scp{v^{0}-\hat{v}}{u_{i}} + \tfrac{1-(1-\omega\sigma_{i}^{2})^{k}}{\sigma_{i}}\scp{r}{w_{i}} 
\end{align*}
\end{proof}
This theorem shows the typical semiconvergence property of the Landweber method: the contributions $\scp{v^{k}-\hat{v}}{u_{i}}$ of the $i$th singular vector have a contribution from the initial error $v^{0}-\hat{v}$ that decays faster for the larger singular values, and another contribution that does converge to $\scp{r}{w_{i}}/\sigma_{i}$ and which is due to the noise.

In a similar way we get for the iterates of the random descent method:
\begin{theorem}\label{thm:conv-singvec-rd}
  
  Let $A\in\RR^{d\times m}$ with right and left singular vectors $(u_{i})$ and $(w_{i})$, respectively, and singular values $\sigma_{i}$. Moreover let $b\in\RR^{m}$ and $\hat{v}\in\RR^{d}$ with $A\hat{v}= b$ and $\tilde b =b+r$, $v^{0}\in\RR^{d}$ and let $\mu_{i}$ be the eigenvalues of the matrix $M = \EE \left( \tfrac{xx^{T}}{\norm{Ax}^{2}} \right)$. Then the iterates $v^{k}$ be the iterates of Algorithm~\ref{alg:rd} with $\tilde b$ instead of $b$ fulfill
\begin{align*}
\EE \left(\scp{v^{k}-\hat{v}}{u_{i}} \right) = (1 - \mu_{i}\sigma_{i}^{2})^{k}\scp{v^{0}-\hat{v}}{u^{i}} + \tfrac{1-(1-\mu_{i}\sigma_{i}^{2})^{k}}{\sigma_{i}}\scp{r}{w_{i}},
\end{align*}
where the expectation is taken with respect to the random choice of $x$ in the $k$th step.
\end{theorem}
Comparing Theorems~\ref{thm:conv-singvec-lw} and~\ref{thm:conv-singvec-rd} we observe that in the case of random descent we get a factor $\mu_{i}$ which depends on $i$ instead of $\omega$. For the Landweber iteration in Theorem~\ref{thm:conv-singvec-lw} one takes $0<\omega<2/\norm{A}^{2}$ and we see that if $\mu_{i}>\omega$ we have that both terms on the right hand side of the estimate converge faster (the first towards zero and the second increases towards $\scp{r}{w_i}/\sigma_{i}$) and thus, the semiconvergence happens faster for random descent than for the Landweber iteration. We will see the practical implications of this result in numerical experiments in Section~\ref{sec:ill-posed-problems}.

\section{Experiments}
\label{sec:experiment}

\subsection{Comparison with other solvers}
\label{sec:comparison}

In this section we compare SGDAS and RD with other solvers for linear system $Av=b$ that meet our requirements, i.e. only with solvers that only need forward evaluations of $A$ and do not need to store a larger number of vectors, namely with TFQMR~(\cite{freund1993tfqmr}) and CGS~(\cite{sonneveld1989cgs}). Note that GMRES also only needs forward evaluations of $A$, but constructs an orthonormal basis that grows with the number of iterations and hence, does not meet our criteria. Both TFQMR and CGS are designed for square systems but do not assume further structure of the operator $A$. However, we can turn both over- and underdetermined systems into square systems by adding rows or columns. Although this is pretty straightforward, we include the details for completeness:\medskip

\begin{description}
\item[Underdetermined systems:] If $A\in\RR^{m\times d}$ and $b\in\RR^{m}$ with $m< d$ we define 
\begin{align*}
\tilde A ;=
  \begin{pmatrix}
    A\\0
  \end{pmatrix}\in\RR^{d\times d},\quad \tilde b :=
  \begin{pmatrix}
    b\\0
  \end{pmatrix}\in\RR^{d}.
\end{align*}
Then $v$ solves $\tilde A v = \tilde b$ exactly if it solves $Av=b$.
\item[Overdetermined systems:] If $m>d$ we define
\begin{align*}
\tilde A :=
  \begin{pmatrix}
    A & 0
  \end{pmatrix}\in\RR^{m\times m}
\end{align*}
and have that a vector $\tilde v = (v^{T}, w^{T})^{T}$ solves $\tilde A\tilde v = b$ exactly if $v$ solves $Av = b$. If the system $Av=b$ is inconsistent we still have that $\tilde v$ solves $\tilde A^{T}A\tilde v = \tilde A^{T}b$ exactly if $v$ solves $A^{T}Av = A^{T}b$.
\end{description}
\medskip

We implemented SGDAS and RD in Python using numpy and scipy. We generated several sparse random $m\times d$ matrices $A$, corresponding solutions $\hat{v}$ and right hand sides $b = A\hat{v}$ to obtain consistent linear systems. We let the methods SGDAS (Algorithm~\ref{alg:sgd}) and RD (Algorithm~\ref{alg:rd}) run with different isotropic random vectors until a specified relative tolerance $\norm{Av-b}/\norm{b}$ was reached or a maximum number of iterations of 10000 has been reached. As a comparison we called TFQMR and CGS from scipy with the same tolerance and maximal number of iterations. In Table~\ref{tab:experiment1a} we collect the results for several sizes and densities of $A$ and a fairly large tolerance of $10^{-2}$ and Table~\ref{tab:experiment1b} gives results for smaller matrices $A$ and a smaller tolerance of $10^{-5}$.

Furthermore, we tested RD on least squares problems from the SuiteSparse Matrix Collection~\cite{davis2011suitesparse}. We used matrices of different sizes and let RD, TFQMR and CGS run for $10\cdot\max(m,d)$ iterations or until a tolerance of  $10^{-2}$ is reached. We did not include spherical sampling in this case since the iterates are the same as for normal sampling (as the update is zero homogeneous and the spherical uniform distribution is the normalized normal distribution). We report the final relative residual and runtime in Table~\ref{tab:suitesparse}.

Notably, RD always works and is reasonably fast. Sometimes it is even in faster than TFQMR and CGS. Moreover, TFQMR and CGS sometimes fail dramatically (actually, they fail on most non-square problems). Moreover, as expected, SGDAS is always slower than RD. In conclusion, RD is a reliable and comparable fast method to minimize least squares functionals under the computational constraints that we consider in this paper. Also the distribution of the random $x$ seems to play a role, but our experiments do not allow for a clear conclusion in this regard.

\begin{table}[htb]
  \centering
  \footnotesize
  \subfloat[Size of $A$: $300\times1200$, density of $A$: $0.1$.]{\begin{tabular}{llrrr}
\toprule
 & & $\tfrac{\norm{Av-b}}{\norm{b}}$ & $\norm{v}$ & time (s)\\
\midrule
Rademacher & SGDAS   &     9.47e-01 &    6.69e-01 & 2.60e+00  \\
 & RD      &     1.00e-02 &    3.42e+01 & 1.04e+00  \\
Coordinate & SGDAS   &     1.00e+00 &    6.58e-04 & 1.96e+00  \\
 & RD      &     9.98e-03 &    3.53e+01 & 8.43e-01  \\
Spherical & SGDAS   &     1.00e+00 &    6.63e-04 & 2.63e+00  \\
 & RD      &     9.97e-03 &    3.41e+01 & 1.15e+00  \\
Normal & SGDAS   &     9.46e-01 &    6.82e-01 & 2.65e+00  \\
 & RD      &     9.98e-03 &    3.38e+01 & 9.77e-01  \\
TFQMR &  &     3.79e+00 &    4.06e+02 & 1.49e+00  \\
CGS   &  &     1.10e+04 &    5.08e+05 & 2.41e+00  \\
\bottomrule
\end{tabular}
}
  
  \subfloat[Size of $A$: $1200\times300$, density of $A$: $0.1$.]{\begin{tabular}{llrrr}
\toprule
 & & $\tfrac{\norm{Av-b}}{\norm{b}}$ & $\norm{v}$ & time (s)\\
\midrule
Rademacher & SGDAS   &     8.69e-01 &    1.97e+00 & 2.42e+00  \\
 & RD      &     9.97e-03 &    1.74e+01 & 1.00e+00  \\
Coordinate & SGDAS   &     9.99e-01 &    8.54e-03 & 2.05e+00  \\
 & RD      &     9.56e-03 &    1.74e+01 & 9.27e-01  \\
Spherical & SGDAS   &     9.99e-01 &    8.52e-03 & 2.20e+00  \\
 & RD      &     9.98e-03 &    1.74e+01 & 8.60e-01  \\
Normal & SGDAS   &     8.69e-01 &    1.96e+00 & 2.09e+00  \\
 & RD      &     9.99e-03 &    1.74e+01 & 8.62e-01  \\
TFQMR &  &     3.16e+00 &    2.78e+02 & 1.44e+00  \\
CGS   &  &     1.96e+01 &    9.89e+02 & 2.36e+00  \\
\bottomrule
\end{tabular}
}

  \subfloat[Size of $A$: $600\times600$, density of $A$: $0.5$.]{\begin{tabular}{llrrr}
\toprule
 & & $\tfrac{\norm{Av-b}}{\norm{b}}$ & $\norm{v}$ & time (s)\\
\midrule
Rademacher & SGDAS   &     9.49e-01 &    6.74e-01 & 8.32e+00  \\
 & RD      &     6.20e-02 &    2.34e+01 & 8.75e+00  \\
Coordinate & SGDAS   &     1.00e+00 &    2.65e-03 & 8.30e+00  \\
 & RD      &     7.79e-02 &    2.25e+01 & 8.38e+00  \\
Spherical & SGDAS   &     1.00e+00 &    2.65e-03 & 8.73e+00  \\
 & RD      &     7.10e-02 &    2.34e+01 & 8.54e+00  \\
Normal & SGDAS   &     9.51e-01 &    6.47e-01 & 8.36e+00  \\
 & RD      &     7.01e-02 &    2.30e+01 & 8.34e+00  \\
TFQMR &  &     4.00e+00 &    1.46e+02 & 4.76e+00  \\
CGS   &  &     1.64e+03 &    7.17e+04 & 8.72e+00  \\
\bottomrule
\end{tabular}
}

  \caption{Comparison of SGDAS, RD, TFQMR and CGS. Stopped at relative tolerance of $10^{-2}$ or after 10000 iterations. Random sparse matrices $A$ with normally distributed entries and random solution vectors $\hat v$ with normally distributed entries.}
  \label{tab:experiment1a}
\end{table}

\begin{table}[htb]
  \centering
  \footnotesize
  \subfloat[Size of $A$: $200\times100$, density of $A$: $0.02$.]{\begin{tabular}{llrrr}
\toprule
 & & $\tfrac{\norm{Av-b}}{\norm{b}}$ & $\norm{v}$ & time (s)\\
\midrule
Rademacher & SGDAS   &     1.16e-02 &    1.00e+01 & 2.71e+01  \\
 & RD      &     1.00e-05 &    1.02e+01 & 7.62e+00  \\
Coordinate & SGDAS   &     5.19e-01 &    3.89e+00 & 2.68e+01  \\
 & RD      &     9.86e-06 &    1.02e+01 & 2.79e+00  \\
Spherical & SGDAS   &     5.19e-01 &    3.89e+00 & 3.19e+01  \\
 & RD      &     1.00e-05 &    1.02e+01 & 7.79e+00  \\
Normal & SGDAS   &     1.24e-02 &    1.00e+01 & 2.38e+01  \\
 & RD      &     1.00e-05 &    1.02e+01 & 6.93e+00  \\
TFQMR &  &     1.12e+00 &    1.15e+01 & 3.84e+01  \\
CGS   &  &     2.26e+13 &    4.01e+16 & 4.44e+01  \\
\bottomrule
\end{tabular}
}

  \subfloat[Size of $A$: $150\times100$, density of $A$: $0.1$.]{\begin{tabular}{llrrr}
\toprule
 & & $\tfrac{\norm{Av-b}}{\norm{b}}$ & $\norm{v}$ & time (s)\\
\midrule
Rademacher & SGDAS   &     9.25e-03 &    1.09e+01 & 3.14e+01  \\
 & RD      &     1.00e-05 &    1.10e+01 & 1.88e+00  \\
Coordinate & SGDAS   &     4.86e-01 &    4.54e+00 & 2.80e+01  \\
 & RD      &     9.98e-06 &    1.10e+01 & 1.38e+00  \\
Spherical & SGDAS   &     4.87e-01 &    4.53e+00 & 3.37e+01  \\
 & RD      &     1.00e-05 &    1.10e+01 & 2.16e+00  \\
Normal & SGDAS   &     8.85e-03 &    1.09e+01 & 2.68e+01  \\
 & RD      &     1.00e-05 &    1.10e+01 & 1.69e+00  \\
TFQMR &  &     3.93e-07 &    5.54e+01 & 1.06e-01  \\
CGS   &  &     nan &    nan & 4.69e+01  \\
\bottomrule
\end{tabular}
}

  \caption{Comparison of SGDAS, RD, TFQMR and CGS. Stopped at relative tolerance of $10^{-5}$ or after 500000 iterations. Random sparse matrices $A$ with normally distributed entries and random solution vectors $\hat v$ with normally distributed entries.}
  \label{tab:experiment1b}
\end{table}

\begin{table}[htb]
  \centering
  \footnotesize
  \begin{tabular}{lcclrr}\toprule Name & $m$ & $d$ & Method & $\tfrac{\norm{Av-b}}{\norm{b}}$  & time (s)\\\midrule\verb|ash292| & 292 & 292 & & &\\ & & & RD Rademacher  &     1.09e-01 &     7.12e-01  \\ & & & RD Coordinate  &     1.11e-01 &     8.63e-01  \\ & & & RD Normal  &     1.15e-01 &     8.40e-01  \\ & & & TFQMR &     2.49e-03 &     1.09e-01  \\ & & & CGS   &     9.50e-03 &     2.73e-02  \\\midrule\verb|ash331| & 331 & 104 & & &\\ & & & RD Rademacher  &     1.00e-02 &     2.61e-01  \\ & & & RD Coordinate  &     5.53e-03 &     4.05e-01  \\ & & & RD Normal  &     9.79e-03 &     2.43e-01  \\ & & & TFQMR &     9.92e-01 &     1.09e+00  \\ & & & CGS   &     6.48e+25 &     1.51e+00  \\\midrule\verb|ash608| & 608 & 188 & & &\\ & & & RD Rademacher  &     9.98e-03 &     9.81e-01  \\ & & & RD Coordinate  &     9.66e-03 &     4.23e-01  \\ & & & RD Normal  &     9.99e-03 &     1.50e+00  \\ & & & TFQMR &     1.00e+00 &     3.03e+00  \\ & & & CGS   &     1.90e+26 &     5.40e+00  \\\midrule\verb|illc1033| & 1033 & 320 & & &\\ & & & RD Rademacher  &     2.95e-02 &     5.72e+00  \\ & & & RD Coordinate  &     3.15e-02 &     5.12e+00  \\ & & & RD Normal  &     2.42e-02 &     6.25e+00  \\ & & & TFQMR &     1.12e+00 &     1.36e+01  \\ & & & CGS   &     1.70e+12 &     1.68e+01  \\\midrule\verb|Maragal_2| & 555 & 350 & & &\\ & & & RD Rademacher  &     3.10e-02 &     3.90e+00  \\ & & & RD Coordinate  &     4.04e-02 &     3.38e+00  \\ & & & RD Normal  &     3.19e-02 &     3.56e+00  \\ & & & TFQMR &     1.47e+00 &     4.62e+00  \\ & & & CGS   &     1.88e+07 &     5.02e+00  \\\midrule\verb|Maragal_3| & 1690 & 860 & & &\\ & & & RD Rademacher  &     2.70e-02 &     4.42e+01  \\ & & & RD Coordinate  &     2.08e-02 &     3.41e+01  \\ & & & RD Normal  &     2.63e-02 &     3.48e+01  \\ & & & TFQMR &     1.76e+00 &     6.16e+01  \\ & & & CGS   &     2.83e+04 &     1.27e+02  \\\midrule\verb|Maragal_4| & 1964 & 1034 & & &\\ & & & RD Rademacher  &     2.52e-02 &     6.86e+01  \\ & & & RD Coordinate  &     1.68e-02 &     6.11e+01  \\ & & & RD Normal  &     2.38e-02 &     6.36e+01  \\ & & & TFQMR &     8.99e+00 &     9.63e+01  \\ & & & CGS   &     1.09e+03 &     1.88e+02  \\\bottomrule\end{tabular}
  \caption{Comparison of RD, TFQMR, and CGS on least squares problems from the SuiteSparse Matrix collection.}
  \label{tab:suitesparse}
\end{table}

\subsection{Ill-posed problems}
\label{sec:ill-posed-problems}

In this section we illustrate the effect of the findings from Section~\ref{sec:illposed}. We consider a simple toy example with $d=m=100$, namely the discrete counterpart of ``inverse integration''. The corresponding map $A$ is the cumulative sum, i.e. $(Av)_{i} = \sum_{j=1}^{i}v_{i}$. From Theorems~\ref{thm:conv-singvec-lw} and~\ref{thm:conv-singvec-rd} we see that random descent with normally distributed directions has advantages over the Landweber iteration if the eigenvalues $\mu_{i}$ of $M$ are larger than the stepsize $\omega$ of the Landweber iteration. In Figure~\ref{fig:invint_factors} we plot the values $\omega\sigma_{i}^{2} = \sigma_{i}^{2}/\norm{A}^{2}$ and $\mu_{i}\sigma_{i}^{2}$. Larger values lead to faster semiconvergence and we observe that the values $\mu_i\sigma_{i}^{2}$ are indeed larger for larger indices $i$, i.e. for the smaller singular values. This indicates, that random descent should perform favorably if the initial error $v^{0}-\hat{v}$ has large contributions from singular vectors from small singular values, i.e. for rough solutions. We constructed a rough solution $b$ and corresponding noisy data $\tilde b = b+r$ with little noise (see Figure~\ref{fig:invint_data}).

\begin{figure}
  \centering
  \includegraphics[width=0.5\textwidth]{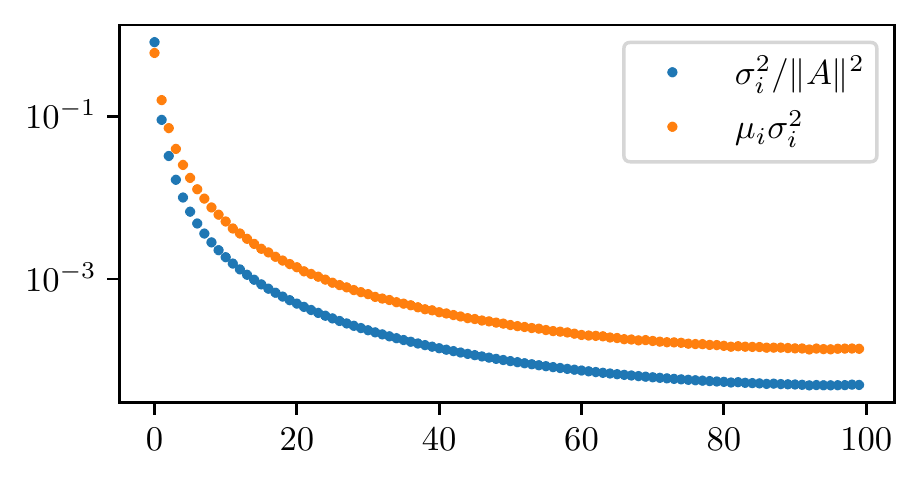}
  \caption{Comparison of factors influencing convergence of the Landweber method ($\sigma_{i}^2/\norm{A}^{2}$) and random descent with normally distributed directions ($\mu_i\sigma_i^2$).}
  \label{fig:invint_factors}
\end{figure}

\begin{figure}
  \centering
  \includegraphics[width=0.7\textwidth]{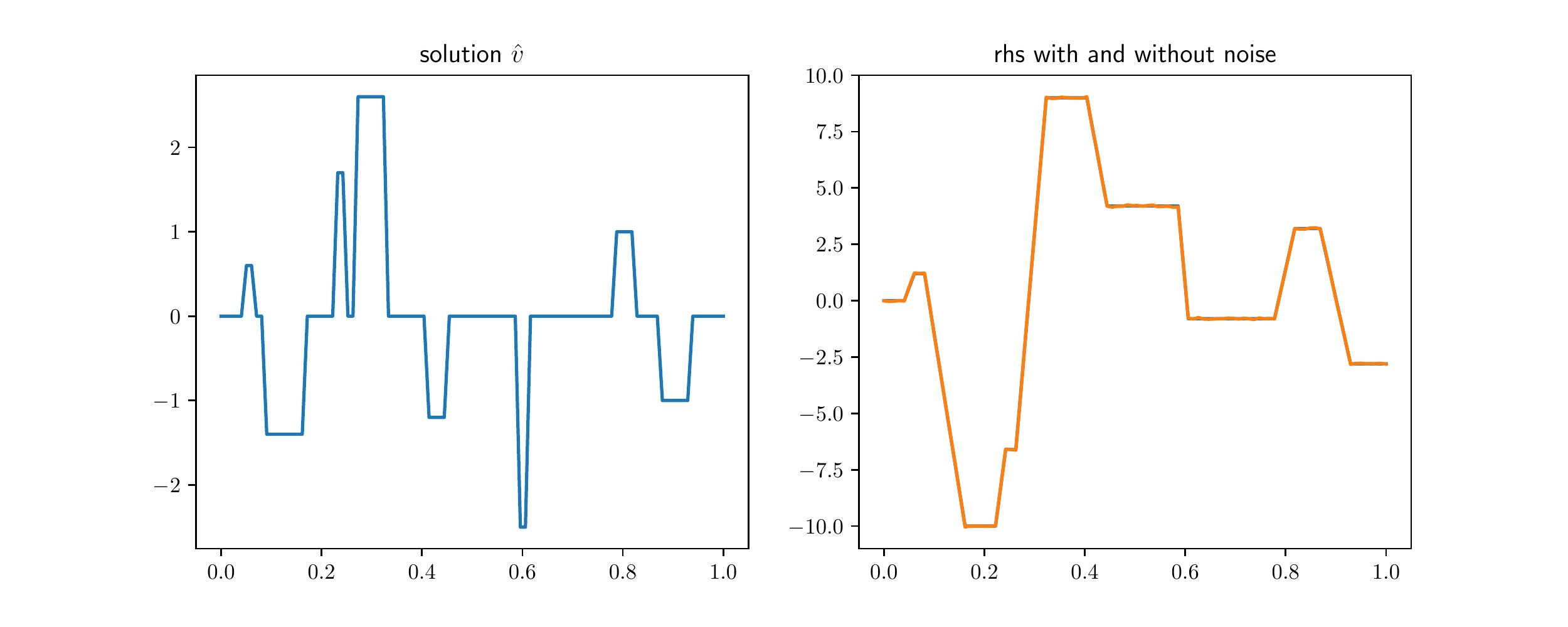}
  \caption{Solution and data for the ill-posed inverse integration example}
  \label{fig:invint_data}
\end{figure}

We have run the Landweber iteration as well as random descent (with different sampling of the direction $x$) and report the decay of residuals and semiconvergence of the errors in Figure~\ref{fig:invint_convergence}. We observe indeed a faster asymptotic decay of the residual for certain sampling schemes in random descent (namely for spherical, normal and Rademacher directions) than for the Landweber iteration, although the residual decays faster for Landweber in the beginning as predicted by our observation in Figure~\ref{fig:invint_factors} and similar for the semiconvergence of the errors.

\begin{figure}
  \centering
  \includegraphics[width=0.7\textwidth]{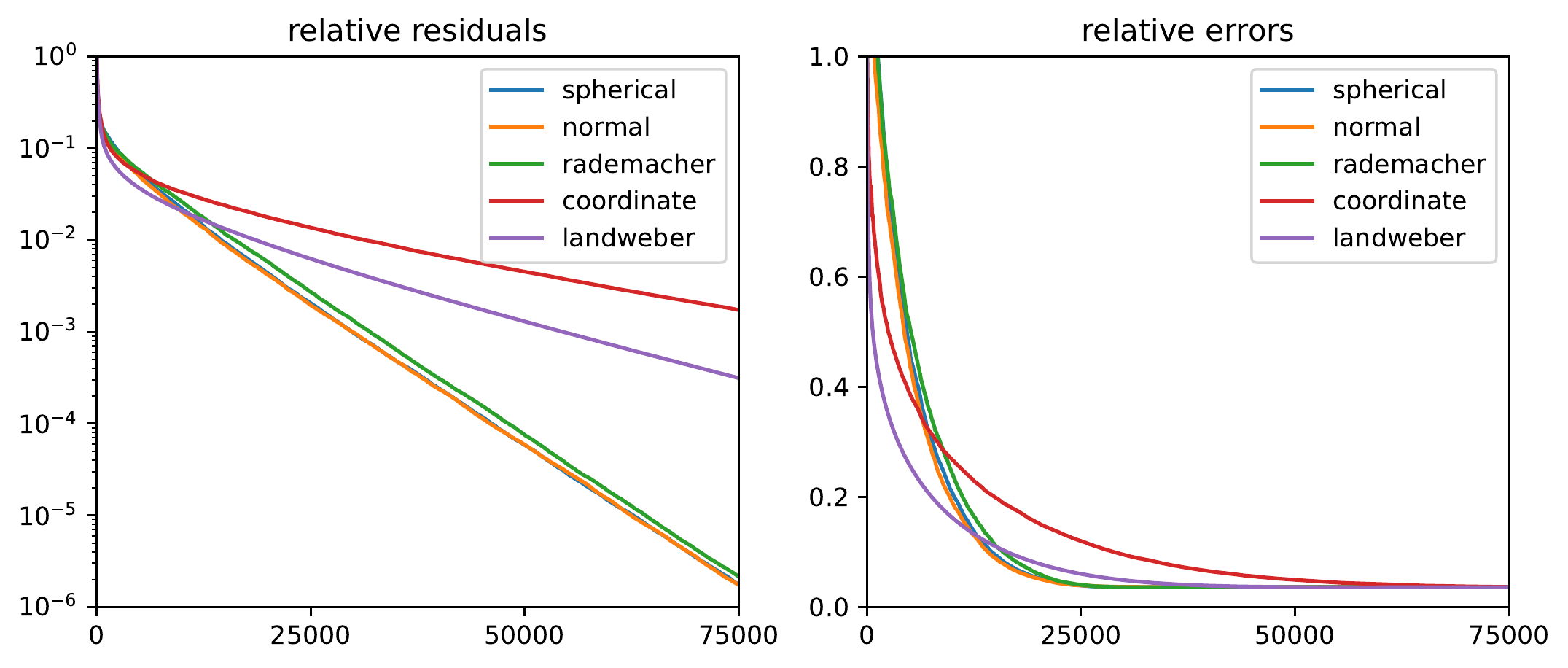}
  \caption{Convergence of residuals and semi-convergence of errors for the ill-posed inverse integration example}
  \label{fig:invint_convergence}
\end{figure}

In Table~\ref{tab:experiment3_illposed} we report values of the best errors achieved by the methods and also the errors achieved with stopping according to Morozov's discrepancy principle. We observe that the adjoint free methods perform comparably well in terms of reconstruction quality, but are in fact faster than the Landweber iteration. 

\begin{table}[htb]
  \centering
 \begin{tabular}{lrrrr}
\toprule
 & Best relative error & iter & Morozov error & stop index\\
\midrule
Spherical  & 0.036 & 33525 &    0.058 &    19038 \\
Normal     & 0.036 & 40526 &    0.057 &    18811 \\
Rademacher & 0.036 & 32714 &    0.054 &    20912 \\
Coordinate & 0.037 & 74983 &    0.053 &    47446 \\
Landweber  & 0.036 & 60421 &    0.052 &    28240 \\
\bottomrule
\end{tabular}

  \caption{Results of the experiment for the ill--posed problem shown in Figure~\ref{fig:invint_convergence}. ``Morozov error'' shows the relative reconstruction error achieved by stopping the iteration with Morozov's discrepancy principle, i.e. when the residual $\norm{Av^{k}-b}$ falls below $c\norm{b-\tilde b}$ with $c=1.001$.}
  \label{tab:experiment3_illposed}
\end{table}

\bibliographystyle{siamplain}
\bibliography{references}

\begin{thebibliography}{10}

\bibitem{brezinski1998transpose}
{\sc C.~Brezinski and M.~Redivo-Zaglia}, {\em Transpose-free {L}anczos-type
  algorithms for nonsymmetric linear systems}, Numerical Algorithms, 17 (1998),
  pp.~67--103.

\bibitem{chan1998transpose}
{\sc T.~F. Chan, L.~de~Pillis, and H.~van~der Vorst}, {\em Transpose-free
  formulations of {L}anczos-type methods for nonsymmetric linear systems},
  Numer. Algorithms, 17 (1998), pp.~51--66,
  \url{https://doi.org/10.1023/A:1011637511962},
  \url{https://doi.org/10.1023/A:1011637511962}.

\bibitem{curtiss1954mc}
{\sc J.~H. Curtiss}, {\em A theoretical comparison of the efficiencies of two
  classical methods and a {M}onte {C}arlo method for computing one component of
  the solution of a set of linear algebraic equations}, New York University,
  Institute of Mathematical Sciences, New York, 1954.
\newblock Rep. IMM-NYU 211.

\bibitem{davis2011suitesparse}
{\sc T.~A. Davis and Y.~Hu}, {\em The university of florida sparse matrix
  collection}, ACM Trans. Math. Softw., 38 (2011),
  \url{https://doi.org/10.1145/2049662.2049663},
  \url{https://doi.org/10.1145/2049662.2049663}.

\bibitem{dimov2000mc}
{\sc I.~Dimov}, {\em Monte {C}arlo algorithms for linear problems}, in
  Proceedings of the 9th {I}nternational {S}ummer {S}chool on {P}robability
  {T}heory and {M}athematical {S}tatistics ({S}ozopol, 1997), vol.~13, 2000,
  pp.~57--77.

\bibitem{engl1996regularization}
{\sc H.~W. Engl, M.~Hanke, and A.~Neubauer}, {\em Regularization of inverse
  problems}, vol.~375, Springer Science \& Business Media, 1996.

\bibitem{forsythe1950mc}
{\sc G.~E. Forsythe and R.~A. Leibler}, {\em Matrix inversion by a {M}onte
  {C}arlo method}, Math. Tables Aids Comput., 4 (1950), pp.~127--129.

\bibitem{freund1993tfqmr}
{\sc R.~W. Freund}, {\em A transpose-free quasi-minimal residual algorithm for
  non-{H}ermitian linear systems}, SIAM J. Sci. Comput., 14 (1993),
  pp.~470--482, \url{https://doi.org/10.1137/0914029},
  \url{https://doi.org/10.1137/0914029}.

\bibitem{gower2015randomized}
{\sc R.~M. Gower and P.~Richt{\'a}rik}, {\em Randomized iterative methods for
  linear systems}, SIAM Journal on Matrix Analysis and Applications, 36 (2015),
  pp.~1660--1690.

\bibitem{halton1994sequential}
{\sc J.~H. Halton}, {\em Sequential {M}onte {C}arlo techniques for the solution
  of linear systems}, Journal of Scientific Computing, 9 (1994), pp.~213--257.

\bibitem{jiao2017preasymptotic}
{\sc Y.~Jiao, B.~Jin, and X.~Lu}, {\em Preasymptotic convergence of randomized
  {K}aczmarz method}, Inverse Problems, 33 (2017), p.~125012.

\bibitem{leventhal2010randomized}
{\sc D.~Leventhal and A.~S. Lewis}, {\em Randomized methods for linear
  constraints: Convergence rates and conditioning}, Mathematics of Operations
  Research, 35 (2010), pp.~641--654.

\bibitem{lorenz2018randomized}
{\sc D.~A. Lorenz, S.~Rose, and F.~Sch{\"o}pfer}, {\em The randomized
  {K}aczmarz method with mismatched adjoint}, BIT Numerical Mathematics, 58
  (2018), pp.~1079--1098.

\bibitem{nesterov2017gradient-free}
{\sc Y.~Nesterov and V.~Spokoiny}, {\em Random gradient-free minimization of
  convex functions}, Found. Comput. Math., 17 (2017), pp.~527--566,
  \url{https://doi.org/10.1007/s10208-015-9296-2},
  \url{https://doi.org/10.1007/s10208-015-9296-2}.

\bibitem{polyak1987}
{\sc B.~T. Polyak}, {\em Introduction to Optimization}, Optimization Software,
  Inc., Publications Division, New York, 1987.

\bibitem{sonneveld1989cgs}
{\sc P.~Sonneveld}, {\em C{GS}, a fast {L}anczos-type solver for nonsymmetric
  linear systems}, SIAM J. Sci. Statist. Comput., 10 (1989), pp.~36--52,
  \url{https://doi.org/10.1137/0910004}, \url{https://doi.org/10.1137/0910004}.

\bibitem{steinerberger2021randomized}
{\sc S.~Steinerberger}, {\em Randomized {K}aczmarz converges along small
  singular vectors}, SIAM Journal on Matrix Analysis and Applications, 42
  (2021), pp.~608--615.

\bibitem{stich2013optimization}
{\sc S.~U. Stich, C.~L. Muller, and B.~Gartner}, {\em Optimization of convex
  functions with random pursuit}, SIAM Journal on Optimization, 23 (2013),
  pp.~1284--1309.

\bibitem{strohmer2009kaczmarz}
{\sc T.~Strohmer and R.~Vershynin}, {\em A randomized {K}aczmarz algorithm with
  exponential convergence}, J. Fourier Anal. Appl., 15 (2009), pp.~262--278,
  \url{https://doi.org/10.1007/s00041-008-9030-4},
  \url{https://doi.org/10.1007/s00041-008-9030-4}.

\bibitem{vershynin2018high}
{\sc R.~Vershynin}, {\em High-dimensional probability: An introduction with
  applications in data science}, vol.~47, Cambridge university press, 2018.

\bibitem{wang2008mc}
{\sc Q.~Wang, D.~Gleich, A.~Saberi, N.~Etemadi, and P.~Moin}, {\em A {M}onte
  {C}arlo method for solving unsteady adjoint equations}, J. Comput. Phys., 227
  (2008), pp.~6184--6205, \url{https://doi.org/10.1016/j.jcp.2008.03.006},
  \url{https://doi.org/10.1016/j.jcp.2008.03.006}.

\end{thebibliography}
\end{document}